\documentclass[10pt]{amsart}
\usepackage{amsmath,amssymb,latexsym,soul,cite,amsthm,color,enumitem,graphicx,tikz, mathtools,microtype}

\usepackage[utf8]{inputenc}
\usepackage{amsfonts,amsmath,amsthm,latexsym,amssymb,marvosym,esint,xfrac,bbm,amssymb}
\usepackage{xcolor}
\usepackage{tikz}
\tikzstyle{mybox} = [draw=black, very thick, rectangle, rounded corners, inner ysep=5pt, inner xsep=5pt]
\usepackage{hyperref}

\usepackage{verbatim}
\usepackage{ulem}
\usepackage[makeroom]{cancel}
\usepackage[margin=1in]{geometry}
 \usepackage[usenames,dvipsnames]{pstricks}
 \usepackage{pst-grad} 
 \usepackage{pst-plot} 
\allowdisplaybreaks

 \numberwithin{equation}{section}
 \newtheorem{theorem}{Theorem}[section]
 
 \newtheorem{lemma}[theorem]{Lemma}

 \newtheorem{definition}[theorem]{Definition}
 \theoremstyle{definition}

\newcommand{\esssup}{\operatornamewithlimits{ess\, sup}}
\newcommand{\essinf}{\operatornamewithlimits{ess\,inf}}
\newcommand{\essosc}{\operatornamewithlimits{ess\, osc}}

\def\R{\mathbb{R}}

\def\d{\mathrm{d}}

\def\e{\epsilon}
\title{Qualitative properties of solutions to parabolic anisotropic equations: Part I - Expansion of positivity}
\author{Simone Ciani {\&} Eurica Henriques {\&} Mariia Savchenko  {\&} Igor I. Skrypnik }

\begin{document}

\begin{abstract} 
We prove expansion of positivity and reduction of the oscillation results to the local weak solutions to a doubly nonlinear anisotropic class of parabolic differential equations with bounded and measurable coefficients, whose prototype is  
\begin{equation*} 
u_t-\sum\limits_{i=1}^N \left( u^{(m_i-1)(p_i-1)} \ |u_{x_i}|^{p_i-2} \  u_{x_i}  \right)_{x_i}=0 ,
\end{equation*} 

\noindent for a restricted range of $p_i$s and $m_i$s, that reflects their competition for the diffusion. The positivity expansion relies on an exponential shift and is presented separately for singular and degenerate cases. Finally we present a study of the local oscillation of the solution for some specific ranges of exponents, within the singular and degenerate cases.

\vspace{0.2cm}


\vspace{0.1cm}


\end{abstract}

\maketitle
	
    \begin{center}
		\begin{minipage}{9cm}
			\small
  \tableofcontents
		 \end{minipage}
	 \end{center}


\addtocontents{toc}{\protect\setcounter{tocdepth}{1}}

\section{Introduction }\label{Sec.1}

\noindent Anisotropic differential equations arise in the study of physical and geometric phenomena where properties vary with direction: this type of differential equations capture directional dependencies in media such as crystals, layered materials, or biological tissues for that they are crucial in fields such as image processing, material science, and fluid dynamics. In the existing literature, which is already extensive, various anisotropic parabolic equations have been considered, depending on whether they relate to the porous medium equation or the parabolic $p$-Laplacean equation. So for anisotropic porous medium equations, the most common evolution equations in the literature associated with fluid flow in an anisotropic medium are
\begin{equation*}
u_t = \sum_{i=1}^n \left( u^{m_i-1} \ u_{x_i}\right)_{x_i} \ , \quad m_i>0, \ \ \forall i=1, \cdots, n
\end{equation*}
and
\begin{equation*}
u_t = \mathrm{div} \left( u^{\gamma(x,t)} Du \right), \quad \mbox{where $\gamma$ is a bounded function.}
\end{equation*}
While for anisotropic equations associated (in some sense) with the p-Laplacian,  very likely due to the mathematical interest and range of applications, there are more records of anisotropic equations, among which we highlight:
\begin{equation*}
u_t = \sum_{i=1}^n \left( |u_{x_i}|^{p_i-2} \ u_{x_i}\right)_{x_i} \ , \quad p_i>1, \ \ \forall i=1,\cdots, n \ ;
\end{equation*}
\begin{equation*}
u_t = \mathrm{div} \left( |Du|^{p(x,t)-2} \ Du \right) \ , \quad \mbox{where $p(x,t)$ is a bounded and continuous function}
\end{equation*}
and also the double-phase equation (also known as the $(p,q)$-growth equation)
\begin{equation*}
u_t =\mathrm{div} \left( |Du|^{p-2} \ Du + a(x,t) |Du|^{q-2} \ Du\right) , \quad 1<p<q \quad \mbox{and $a$ a nonnegative continuous function.}
\end{equation*}

\noindent In \cite{Iwonabook}, several results related to the existence of solutions for differential equations in anisotropic Musielak-Orlicz spaces (as well as some applications) can be found. Also, in \cite{AS-book2}, \cite{AS05}, \cite{Songexist01} and \cite{Songuniq01}, one can find results on the existence and uniqueness of solutions. Regarding qualitative (and quantitative) properties of the solutions, in addition to the aforementioned works, we may also consider the (non-exhaustive) list of other contributions \cite{AMS}, \cite{AS2}, \cite{BBLV23}, \cite{Buryachenko}, \cite{Cianchi}, \cite{CG}, \cite{CH25}, \cite{CHS24}, \cite{CHS25}, \cite{CMV}, \cite{CVV}, \cite{Tedeev-CP}, \cite{DMV}, \cite{Vaz2}, \cite{Vaz1}, \cite{EH21}, \cite{EH11}, \cite{EH08}, \cite{HU06}, \cite{Ok20}, \cite{LY}, \cite{Marcellini}, \cite{SongJian06}, \cite{SongJian05}, as well as the references contained therein.

\vspace{.2cm}

\noindent In this paper we are concerned with general doubly nonlinear anisotropic parabolic equations not only for the growing interest on them and their numerous applications but also for their inherent mathematical challenges and properties. The prototype equation related to this double anisotropic setting is given by 
\[ u_t-\sum\limits_{i=1}^N \left( u^{(m_i-1)(p_i-1)} \ |u_{x_i}|^{p_i-2} \  u_{x_i}  \right)_{x_i}=0 , \quad \forall \  m_i>0, \ \ p_i>1\]
and defined over the space-time domain $\Omega_T:=\Omega \times (0, T)$, where  $\Omega$ is a bounded domain in $\mathbb{R}^N$ and $T$ a real positive number. We will study local regularity properties of the non-negative solutions to the equation
\begin{equation}\label{eq1.1}
u_t-\sum\limits_{i=1}^N \left(a_i(x, t, u, D u) \right)_{x_i}=0,\quad (x,t)\in \Omega_T,
\end{equation} 
where the functions $a_i(x, t, u, Du)$, $i=1, ..., N$, satisfy the Caratheodory conditions and the following structure conditions for almost every $(x,t) \in \Omega_T$,
\begin{equation}\label{eq1.2}
\begin{cases}
 \sum\limits_{i=1}^N a_i(x, t, u, D u) \ (u^{m^-})_{x_i}  \geqslant K_1  \sum\limits_{i=1}^N u^{(m_i-m^-)(p_i-1)} |(u^{m^-})_{x_i}|^{p_i}, \quad 
m^{-}:=\min( m_1, ..., m_N)\,,\\[.8em]
|a_i(x, t, u, D u)|\leqslant K_2 \;  u^{\frac{(m_i-m^-)(p_i-1)}{p_i}} \left(\sum\limits_{j=1}^N u^{(m_j-m^-)(p_j-1)}|(u^{m^-})_{x_j}|^{p_j}\right)^{1-\frac{1}{p_i}}\,
\end{cases}
\end{equation}
being $K_1$, $K_2$ fixed positive constants, $Du=(u_{x_1}, \cdots, u_{x_N})$ the weak gradient, and all the (anisotropic) exponents satisfy $p_i >1, m_i >0$.
\vskip0.2cm \noindent 

\noindent The approaches adopted by the above mentioned authors rely on a specific notion of solution and, if not for any other reason but that one, it brings the subject into different lights. In this work we study \eqref{eq1.1}-\eqref{eq1.2} under the following definition of weak solution.
\begin{definition} \label{defweaksol}
We say that a non-negative function $u$ satisfying\footnote{See Section \ref{sec:functional-setting} for more details on the latter space.}
\[ u\in C_{loc}(0,T; L^{1+m^-}_{loc}(\Omega)), \quad  u^{m^-}\in L^{\vec{p}}_{loc}(0, T; W^{1, \vec{p}}_{loc}(\Omega)), \]
 and for $i=1,..., N$,
\[u^{\frac{m_i(p_i-1)+m^{-}}{p_i}}, \left(u^{\frac{m_i(p_i-1)+m^{-}}{p_i}} \right)_{x_i} \in L^{p_i}_{loc}(0,T; L^{p_i}_{loc}(\Omega)),\quad  \]
 is a local weak sub(super)-solution to \eqref{eq1.1} if the following integral inequalities hold
\begin{equation}\label{eq1.3}
\int\limits_{K} u\,\varphi \,dx \Big|^{t_2}_{t_1} +\int\limits^{t_2}_{t_1}\int\limits_{K}\Big[-u\,\varphi_t+\sum\limits_{i=1}^N a_i(x, t, u, Du)\,\varphi_{x_i}\Big]\,dx\,dt\leqslant (\geqslant) \,0,
\end{equation}
for any compact $K\subset \Omega$, any interval $[t_1, t_2]\subset (0, T)$ and any non-negative $\varphi$, such that \[\varphi\in  W^{1,1+m^-}_{loc}(0,T; L^{1+m^-}(K)) \cap L^{\vec{p}}(0,T, W^{1,\vec{p}}_0(K)),\quad \vec{p}=(p_1, ..., p_N). \]
\noindent We say that $u$ is a non-negative, local weak solution to equation \eqref{eq1.1} if $u$ is both a non-negative, local weak sub- and super-solution.
\end{definition}

\noindent Observe that the regularity assumptions on $u$ assure that all the integrals appearing in \eqref{eq1.3} are well defined. In particular these regularity assumptions will allow us to consider test functions $\varphi=(u^{m^-}- k^{m^-})_{\pm} \xi^{p_+}$ and  guarantee that all the integrals appearing in the energy estimates \eqref{eq2.3} are finite as well as the integral evolving space derivatives related to the anisotropic embedding \eqref{embedding}. 

\vspace{.2cm}

\noindent The choice of $m^-:=\min( m_1, ..., m_N)$ allows us not only to recover what was presented in previous works regarding isotropic settings (see for instances \cite{BHSS21}) but also to obtain better estimates in the framework of local boundedness. Note that another possibility was to consider $m^-:=\min(1, m_1, ..., m_N)$, and in such case one recovers, for the isotropic setting, the setting and results presented in \cite{DBGV-mono}.

\vspace{.2cm}

\noindent Above all, the objective of this work is to make a consistent and precise contribution to the study of anisotropic differential equations, particularly with regard to results leading to local continuity.  In section~\ref{Sec.4}, results on the expansion of
positivity of $u$ are proved under the umbrella of an exponential shift (the isotropic siblings can be found in \cite{DBGV-mono} and also in \cite{EH22} for the doubly nonlinear setting). Nonetheless, there is a price to be paid: the exponents $\lambda_i= m_i(p_i-1)$ must be kept close, that is,  $\displaystyle{\max_{i} \lambda_i-\min_{i} \lambda_i \leqslant \epsilon_{\star}}$, for some small positive $\epsilon_{\star}$ (determined along the proofs and) depending only on the data. In section~\ref{Sec.6}, based on the previous results regarding the expansion of positivity and studying an alternative argument considered within an intrinsic geometry dictated by  \eqref{eq1.1}, the reduction of the oscillation of the nonnegative, locally bounded, local weak solutions to \eqref{eq1.1}-\eqref{eq1.2}, for specific ranges of $\lambda_i$ and $p_i$, is derived both for the fast diffusion range and the slow diffusion range.  

\subsection{The Result} For numbers $\{p_i\}_{i=1}^N$ and $\{m_i\}_{i=1}^N$ related to the degeneracy of the equation \eqref{eq1.1}, let us define for $i=1, \cdots, N$ the numbers
\[\lambda_-:=\min\limits_{1\leqslant i\leqslant N}\lambda_i,\quad \lambda_+:=\max\limits_{1\leqslant i\leqslant N}\lambda_i,\quad \lambda_i:=m_i(p_i-1)\, .  \] Number $\lambda_{+}$ discriminates the {\it singular} or {\it degenerate} behavior of the equation\footnote{The case $\lambda_+=1$ is a special one, reminiscent of the Trudinger's equation -since the parabolic and elliptic  energy terms get balanced - and will be treated in a forthcoming work.}, when respectively $\lambda_+ >1$ or $\lambda_+<1$. Our main results, Theorem \ref{th4.1} and Theorem \ref{th5.1}, state that, for non-negative local weak super-solutions to \eqref{eq1.1}-\eqref{eq1.2}, there exists a number $\epsilon^*\in (0,1)$ determined a priori only on the data, such that, if $\lambda_+-\lambda_{-}<\epsilon^*$ and the measure information
\[|\mathcal{B}_r(y)\cap [u(\cdot, s)\geqslant k]|\geqslant \alpha_0 |\mathcal{B}_r(y)|,\qquad \mathrm{for} \ \ \mathrm{some} \ \ \alpha_0 \in (0,1)\]
is available at some time level $s$ within anisotropic sets $\mathcal{B}_r$ whose geometry depends on $k$, for constants $k,r, \alpha_0$, then there exist constants $\epsilon_{*}\in (0,1)$, $C_{*}>1$ and $\delta_1, \delta_2\in (0,1)$, depending only on the data and $\alpha_0$, such that
\[u(x, t)\geqslant \frac{k}{C_{*}},\quad x\in \mathcal{B}_{2r}(y),\]
for all $t$ in a time range $s+ \delta_1 \,r^{p_{+}}\,\,k^{1-\lambda_{+}} \leqslant t\leqslant s+ \delta_2 \,r^{p_{+}}\,\,k^{1-\lambda_{+}}$ that is different from degenerate to singular case. See Sections \ref{Sec.4}-\ref{Sec.5} for the precise statements. The phenomenon of expansion of positivity lies at the heart of several other results towards regularity, such as results on the reduction of the oscillation and any form of Harnack inequality, and therefore, of H\"older estimates. In the last section of this work, we show its application to the control of the local oscillation of solutions.

\subsection{The Method} Our method is reminiscent of the elliptic technique of \cite{DiBe-Remarks} where, roughly speaking, an expansion of positivity is found for competing exponents in an intrinsic geometry whose magnitude is chosen accordingly to the degeneration of the equation. In this work we adapt the aforementioned strategy to the parabolic doubly nonlinear case, through a careful adaptation of the so-called parabolic technique of the exponential-shift, devised in \cite{Acta}. The simple idea underlying this technique is that, when looking at the equation by the magnifying glass provided by an exponential transformation, the degeneracy (fast or slow propagation) caused by a time factor $t^{-\bar{\lambda}}$ gets reduced and the elliptic terms dominate the diffusion, allowing us to perform De Giorgi's level set machinery \cite{DG}. Our approach has the advantage that can be used for operators with full-quasilinear structure (see \eqref{eq1.2}), and therefore it brings to light the structural local behavior of solutions to equations that do not, necessarily, enjoy a comparison principle. On the other hand, the disadvantage of our approach is that the exponential-shift encumbers so much the nonlinear anisotropy, that the control on the gap $p_{\text{max}}-p_{\text{min}}$ allowed for the expansion of positivity is fixed, but only qualitative.

\section{Notation, functional setting and auxiliary results}\label{Sec.2}

\subsection{Notation} 

\noindent Along the text we will be considering several exponents related to the anisotropic ones $m_i$ and $p_i$, for $i=1, \cdots, N$, and some anisotropic domains. More precisely, we will consider 
\begin{itemize}
    \item Exponents 

    \[\lambda_-:=\min\limits_{1\leqslant i\leqslant N}\lambda_i,\quad \lambda_+:=\max\limits_{1\leqslant i\leqslant N}\lambda_i,\quad \lambda_i:=m_i(p_i-1),\quad i=1, \cdots, N \]

    \[ p_{+} = \max\limits_{1\leqslant i\leqslant N} p_i , \quad p_{-} = \min\limits_{1\leqslant i\leqslant N} p_i\]
    \[\frac{1}{p}= \frac{1}{N} \sum_{i=1}^N \frac{1}{p_i} , \quad p <N \qquad \mbox{and} \qquad 
    \alpha= \sum_{i=1}^N \frac{\alpha_i}{N} , \quad \alpha_i>0\,,\]
    \[ p^\star= \frac{Np}{N-p} \qquad \mbox{and} \qquad p^\star_{\alpha}= p^\star \ \alpha\,.\]

\item Cubes and cylinders 
\[ K_{\vec{r}}(x_0):=\{x: |x_i-x_{i,0}|<r_i, \ \ i=1, ...,N\} \]

\[ Q^{+}_{\vec{r}, \eta}(x_0, t_0):=K_{\vec{r}}(x_0)\times (t_0, t_0+\eta),\quad Q^-_{\vec{r}, \eta}(x_0, t_0):=K_{\vec{r}}(x_0)\times(t_0-\eta, t_0), \]
being 
\[Q_{\vec{r}, \eta}(x_0, t_0):=Q^{+}_{\vec{r}, \eta}(x_0, t_0)\cup Q^{-}_{\vec{r}, \eta}(x_0, t_0) \]
and $\vec{r}:=(r_1, ..., r_N)$, for $r_1,\cdots, r_N, \eta$ real positive fixed numbers.

\item {\it Intrinsic} anisotropic cubes

\[ K^k_{r}(y):=\left\{x:|x_i-y_i|<\frac{r^{\frac{p_+}{p_i}}}{k^{\frac{\lambda_+-\lambda_i}{p_i}}},\ \ i=1, ..., N\right\} , \qquad \mbox{for} \ r,k>0 \ . \]

\end{itemize}

\vspace{.3cm}

\noindent We will consider as structural data the parameters $N$, $p_1$, ..., $p_N$, $m_1$, ..., $m_N$, $K_1$, $K_2$, and we will consider constants $\gamma$ (which may change from line to line) that are  quantitatively determined a priory in terms of the above quantities. 

\subsection{Functional anisotropic setting} \label{sec:functional-setting}

\noindent We define the anisotropic spaces of locally integrable functions as
\[W^{1,\vec{p}}_{loc}(\Omega)= \{u \in W^{1,1}_{loc}(\Omega)\, |\, \,  u_{x_i} \in L^{p_i}_{loc}(\Omega) \},\]
\[L^{\vec{p}}_{loc}(0,T; W^{1, \vec{p}}_{loc}(\Omega))= \{u \in L^1_{loc}(0,T; W^{1,1}_{loc}(\Omega)\, |\, \, u_{x_i} \in L_{loc}^{p_i}(0,T; L^{p_i}_{loc} (\Omega))\},\]
and the respective spaces of functions with zero boundary data
\[W^{1,\vec{p}}_{o}(\Omega)= \{u \in W^{1,1}_{o}(\Omega)\, |\, \, u_{x_i} \in L^{p_i}_{loc}(\Omega) \},\]
\[L^{\vec{p}}_{loc}(0,T; W^{1, \vec{p}}_{o}(\Omega))= \{u \in L^1_{loc}(0,T; W^{1,1}_{o}(\Omega)\, |\, \,  u_{x_i} \in L_{loc}^{p_i}(0,T; L^{p_i}_{loc} (\Omega))\}.\]
\noindent It is known (see  \cite{IlinBesovNikolski}, \cite{russian1}) that when $p>N$ the embedding $W^{1,\vec{p}}(\Omega) \hookrightarrow C^{0,\alpha}_{loc}(\Omega)$ for $\Omega$ regular enough. Therefore in this work we will consider $p<N$.

\vspace{.2cm}

\noindent We will use Lemma \ref{embedding} that provides a Sobolev embedding within the setting of anisotropic parabolic PDEs \cite{DMV}
\begin{lemma} \label{embedding} 
Let $\Omega\subseteq \mathbb{R}^N$ be a rectangular domain, $p<N$, $\alpha_i>0$, for all $i=1, \cdots, N$, and $\sigma\in [1, p^{*}_{\alpha}]$. For any number $\theta\in [0,\,  p/  p^{*}]$ define 
\[
\delta=\delta(\theta, {\bf p}, \alpha)=\theta \,  p^{*}_{\alpha}+\sigma\, (1-\theta),
\]
Then there exists a positive constant $c=c(N, {\bf p}, \alpha, \theta, \sigma)>0$ such that
\begin{equation}
\label{PS}
\iint_{\Omega_{T}}|\varphi|^{\delta}\, dx\, dt\leq c\, T^{1-\theta\, \frac{ p^{*}}{ p}}\left(\sup_{t\in (0, T]}\int_{\Omega}|\varphi|^{\sigma}(x, t)\, dx\right)^{1-\theta}\prod_i \left(\iint_{\Omega_{T}}| (\varphi^{\alpha_i})_{x_i}|^{p_{i}}\, dx\, dt\right)^{\frac{\theta\,   p^{*}}{N \, p_{i}}},
\end{equation} 
for any $\varphi\in L^{1}(0, T; W^{1,1}_{o}(\Omega))$, being the inequality trivial when the right-hand side is unbounded.
\end{lemma}

\subsection{Auxiliary results}

\vspace{.2cm}

\noindent In this section we present several auxiliary results starting with the so called local clustering lemma \cite{DBGV06}.


\begin{lemma}\label{lem2.2}
Let $K_{r}(y)$ be a cube in $\mathbb{R}^{N}$, of edge $r$ centered at $y$, and let $u\in W^{1,1}(K_{r}(y))$ satisfies
\begin{equation}\label{eq2.1}
||(u-k)_{-}||_{W^{1,1}(K_{r}(y))} \leqslant \mathcal{K}\,k\,r^{N-1},\,\,\,\,\,\,and\,\,\,\,\,\, |\{K_{r}(y) : u\geqslant k \}|\geqslant \beta |K_{r}(y)|,
\end{equation}
with some $\beta \in (0,1)$, $k\in\mathbb{R}$ and $\mathcal{K} >0$. Then for any $\xi \in (0,1)$ and any $\nu\in (0,1)$ there exists $\bar{x} \in K_{r}(y)$ and $\delta=\delta(N) \in (0,1)$ such that
\begin{equation}\label{eq2.2}
|\{K_{\bar{r}}(\bar{x}):  u\geqslant \xi\,k \}| \geqslant (1-\nu) |K_{\bar{r}}(y)|,\,\,\, \bar{r}:=\delta \beta^{2}\frac{(1-\xi)\nu}{\mathcal{K}}\,r.
\end{equation}
\end{lemma}

\vspace{.1cm}

\noindent Next result gives an integral estimates which are a key tool in the derivation of regularity results.

\paragraph{Local Energy Estimates}

\begin{lemma}\label{lem2.3}
Let $u$ be a non-negative, local weak sub(super)-solution to equation \eqref{eq1.1} under conditions \eqref{eq1.2}. Then there exists a positive constant $\gamma$, depending only on the data, such that for every cylinder $Q_{\vec{r}, \eta}(y, \tau)\subset \Omega_T$, every $k\in \mathbb{R}_+$, and every piecewise smooth cutoff function $\zeta \in [0,1]$, vanishing on $\partial K_{\vec{r}}(y)$, there holds

$\displaystyle{\sup\limits_{\tau-\eta\leqslant t \leqslant \tau+\eta}\int\limits_{K_{\vec{r}}(y)} g_{\pm}(u^{m^{-}}, k^{m^{-}})\,\zeta^{p_{+}}\,dx +\gamma^{-1} \sum\limits_{i=1}^{n} \iint\limits_{Q_{\vec{r},\eta}(y, \tau)}u^{(m_i-m^{-})(p_i-1)}\left|\left((u^{m^{-}}-k^{m^{-}})_{{\pm}}\right)_{x_i}\right|^{p_{i}}\zeta^{p_{+}}\,dx\,dt}$
\begin{eqnarray}\label{eq2.3}
& \leqslant & \int\limits_{K_{\vec{r}}(y)\times\{\tau-\eta\}}g_{\pm}(u^{m^{-}}, k^{m^{-}}) \,\zeta^{p_{+}}\,dx+\gamma\iint\limits_{Q_{\vec{r},\eta}(y, \tau)}g_{\pm}(u^{m^{-}}, k^{m^{-}}) \,|\zeta_t|\,dx dt \nonumber \\
& & +\gamma \sum\limits_{i=1}^{n}\iint\limits_{Q_{\vec{r},\eta}(y, \tau)}u^{(m_i-m^{-})(p_i-1)}(u^{m^{-}}-k^{m^{-}})_{\pm}^{p_{i}}|\zeta_{x_i}|^{p_{i}}\,dx\,dt
\end{eqnarray}
where \[g_{\pm}(u^{m^{-}}, k^{m^{-}})=  \int\limits_{0}^{(u^{m^{-}}-k^{m^{-}})_{\pm}} s(k^{m^{-}}\pm s)^{\frac{1-m^{-}}{m^{-}}}\,ds\ .\]
\end{lemma}
\begin{proof}
Test \eqref{eq1.3} by $\pm(u^{m^{-}}-k^{m^{-}})_{\pm}\zeta^{p_{+}}$ and integrate over $K_{\vec{r}}(y)\times(\tau-\eta, t)$ with $t\in (\tau-\eta, \tau+\eta)$.  The use of such a test function is justified, modulus a standard
Steklov averaging process, by making use of the alternate weak formulation \eqref{eq1.4}. Using conditions \eqref{eq1.2} and the Young's inequality we arrive at the required \eqref{eq2.3}.
\end{proof}

\vspace{.1cm}

\noindent Following the lead of DeGiorgi, adapted to the parabolic setting by the russian school and then consolidated to degenerate and singular parabolic equations by DiBenedetto and his co-authors, we present and prove De Giorgi-type results. 

\vspace{.1cm}

\noindent {\bf A De Giorgi-type Lemma} - Consider three fixed numbers $\xi$, $a \in (0, 1)$ and $k>0$.

\begin{lemma}\label{lem2.4}
Let $u$ be a non-negative, local weak super-solution to \eqref{eq1.1} under conditions \eqref{eq1.2}. There exists $\nu_-\in (0, 1)$
depending only on the data, $a$, $\vec{r}$, $\eta$ {\color{teal} and $k$}, such that if
\begin{equation}\label{eq2.4}
|Q_{\vec{r},\eta}(y, \tau)\cap\{u\leqslant k \}|\leqslant \nu_- |Q_{\vec{r},\eta}(y, \tau)|,
\end{equation}
then
\begin{equation}\label{eq2.5}
u(x,t)\geqslant a k,\quad (x,t)\in Q_{\vec{r}/2, \eta/2}(y, \tau).
\end{equation}
\end{lemma}

\begin{proof} For $j=0,1,2, ...$ define
 \[ k_j^{m^{-}}=[a k]^{m^{-}}+(1-a^{m^{-}}) k^{m^{-}} 2^{-j} , \] 
\[ \vec{r}_j:=\frac{\vec{r}}{2}(1+2^{-j}) \ ,  \quad \eta_j:=\frac{\eta}{2}(1+2^{-j}) ,\] 
construct the cubes and cylinders
\[ K_j:=K_{\vec{r}_j}(y) , \quad Q_j:=K_j\times I_j, \ \ \mbox{where} \ \ I_j:= (\tau-\eta_j, \tau+\eta) .\] 
and consider piecewise smooth cutoff functions $\zeta_j:=\zeta^{(1)}_j(x)\,\zeta^{(2)}_j(t)$, where 
\begin{itemize}
    \item[] $\zeta^{(1)}_j(x)\in C^1_0(K_j)$; $0\leqslant \zeta^{(1)}_j(x)\leqslant 1$;  $\zeta^{(1)}_j(x)=1$ in $K_{j+1}$ and $|(\zeta^{(1)}_j)_{x_i}|\leqslant \dfrac{\gamma 2^{j}}{r_i}$, for $i=1, ..., N$
    \item[] $\zeta^{(2)}_j(t)\in C^1(\mathbb{R})$; $0\leqslant \zeta^{(2)}_j(t)\leqslant 1$; $\zeta^{(2)}_j(t)=0$ for $t \leqslant \tau-\eta_{j}$ and for $t \geqslant\tau-\eta_{j}$; $\zeta^{(2)}_j(t)=1$ for 
$\tau-\eta_{j+1} \leqslant t\leqslant \tau+\eta_{j+1}$ and $|\zeta^{(2)}_{j, t}|\leqslant \dfrac{\gamma 2^j}{\eta}$.
\end{itemize}  
From Lemma \ref{lem2.3}, taken for the test functions $\varphi(x,t)= (u^{m^{-}}-k^{m^{-}}_j)_{-} \zeta_j^{p_+} $ considered defined in $Q_j$, we obtain
\begin{equation*}\label{est}
\sup\limits_{t \in I_j}\int\limits_{K_j \times\{t\}}g_{-}(u^{m^{-}}, k^{m^{-}}_j)\,\zeta^{p_{+}}_j\,dx+\gamma^{-1}\sum\limits_{i=1}^N \iint\limits_{Q_j}u^{(m_i-m^{-})(p_i-1)}\left| \left((u^{m^{-}}-k^{m^{-}}_j)_{-}\right)_{x_i} \right|^{p_{i}}\,\zeta^{p_{+}}_j dx\,dt 
\end{equation*}
\begin{eqnarray} 
    & \leqslant & \frac{\gamma 2^{j}}{\eta}\iint\limits_{Q_j} g_{-}(u^{m^{-}}, k_j^{m^{-}})\,\,dx dt +\gamma \sum\limits_{i=1}^{N} \frac{2^{j\ p_i}}{r^{p_i}_i} \iint\limits_{Q_j}u^{(m_i-m^{-})(p_i-1)}(u^{m^{-}}-k^{m^{-}}_j)_{-}^{p_{i}}\,dx\,dt \\ \nonumber
    & \leqslant & \gamma 2^{j p_+}\left\{ \frac{k^{m^{-}+1}}{\eta} + \sum\limits_{i=1}^{N} \frac{k^{\lambda_i+m^{-}}}{r^{p_i}_i} \right\} |Q_j \cap[u<k_j]| \\ \nonumber 
    & \leqslant & \gamma 2^{j p^+}k^{\lambda_{+}+m^{-}}\Big(\frac{k^{1-\lambda_{+}}}{\eta}+\sum\limits_{i=1}^N \frac{k^{\lambda_i-\lambda_{+}}}{r^{p_{i}}_i}\Big)|Q_j\cap [u < k_j]| \label{est}
\end{eqnarray} 
by making use of the conditions on $\zeta$ and by estimating $ \displaystyle{g_{-}(u^{m^{-}}, k^{m^{-}}) \leq k_j^{m^{-}+1} \chi_{[u<k_j]}}$. To estimate the left-hand side of the previous integral inequality we define
$$v:=\max(u, a k).$$
By noting that
\[ 
\int\limits_{K_j\times\{t\}}g_{-}(u^{m^{-}}, k_j^{m^{-}}) \,\zeta_{j}^{p_{+}}\,dx \geqslant
\int\limits_{K_j\times\{t\}}g_{-}(v^{m^{-}}, k_j^{m^{-}}) \,\zeta_{j}^{p_{+}}\,dx \geqslant \frac{(a k)^{1-m^{-}}}{2} \int\limits_{K_j\times\{t\}}(v^{m^{-}}-k^{m^{-}}_j)_{-}^2\,\zeta^{p_{+}}_j\,dx,
\] 
and taking $\alpha_i=\dfrac{\lambda_i +m^{-}}{m^{-}\,p_i}$, 
\[ \iint\limits_{Q_j} u^{(m_i-m^{-})(p_i-1)}\left| \left((u^{m^{-}}-k^{m^{-}}_j)_{-}\right)_{x_i} \right|^{p_{i}} \,\zeta_j^{p_{+}}\,dxdt \geqslant \iint\limits_{Q_j \cap [u> a k]} u^{(m_i-m^{-})(p_i-1)}\left| \left((u^{m^{-}}-k^{m^{-}}_j)_{-}\right)_{x_i} \right|^{p_{i}}\,\zeta_j^{p_{+}}\,dxdt\]
\[= \iint\limits_{Q_j \cap [u> a k]} v^{(m_i-m^{-})(p_i-1)}\left| \left((u^{m^{-}}-k^{m^{-}}_j)_{-}\right)_{x_i} \right|^{p_{i}} \,\zeta_j^{p_{+}}\,dxdt \geqslant \iint\limits_{Q_j} \frac{a^{(m_i-m^{-})(p_i-1)}}{\alpha_i^{p_i}} \left|\left((v^{m^{-}}-k^{m^{-}}_j)^{\alpha_i}_{-}\right)_{x_i} \right|^{p_{i}}\,\zeta^{p_{+}}_j\,dx\,dt \ \]
we arrive at 
\begin{multline*}
(a k)^{1-m^{-}}\sup\limits_{t\in I_j}\int\limits_{K_j\times\{t\}}(v^{m^{-}}-k^{m^{-}}_j)_{-}^2\,\zeta^{p_{+}}_j\,dx+ \gamma^{-1} a^{(m^+-m^-)(p^+-1)}\sum\limits_{i=1}^N \iint\limits_{Q_j}\left|\left((v^{m^{-}}-k^{m^{-}}_j)^{\alpha_i}_{-}\right)_{x_i} \right|^{p_{i}} \; \zeta^{p_{+}}_j\,dx\,dt\\
\leqslant \gamma 2^{j p^+}\ k^{\lambda_{+}+m^{-}}\Big(\frac{k^{1-\lambda_{+}}}{\eta}+\sum\limits_{i=1}^N \frac{k^{\lambda_i-\lambda_{+}}}{r^{p_{i}}_i}\Big)|Q_j\cap [u < k_j]|.
\end{multline*}
On the one hand, by applying H\"{o}lder's inequality together with the Sobolev embedding \ref{embedding}, with $\alpha_i=\frac{\lambda_i +m^{-}}{m^{-}\,p_i}$, $\alpha=\frac{1}{N}\sum\limits_{i=1}^n\alpha_i$, $\theta=p/p^{\star}$ and $\sigma=2$, we get
\begin{eqnarray*}
\iint\limits_{Q_j}\left((v^{m^{-}}-k^{m^{-}}_j)_{-}\zeta^{p_{+}}_j\right)^{\alpha p}\,dxdt & \leqslant & \Big(\iint\limits_{Q_j}\big[(v^{m^{-}}-k^{m^{-}}_j)_{-}\zeta^{p_{+}}_j\big]^{p\frac{\alpha N+2}{N}}\,dxdt\Big)^{\frac{N}{N+2/\alpha}}|Q_j\cap\{v<k_j\}|^{\frac{2/\alpha}{N+2/\alpha}}\\ & \leqslant &
\gamma\Big(\sup\limits_{t\in I_j}\int\limits_{K_j\times\{t\}}(v^{m^{-}}-k_j^{m^{-}})^{2}_{-}\zeta^{p_{+}}_j\,dx\Big)^{\frac{p}{N+2/\alpha}} \\
& & \times \Big(\sum\limits_{i=1}^N \iint\limits_{Q_j}\left|\left((v^{m^{-}}-k^{m^{-}}_j)^{\alpha_i}_{-}\right)_{x_i} \right|^{p_{i}}\,dxdt\Big)^{\frac{N}{N+2/\alpha}}|Q_j\cap [u< k_j]|^{\frac{2/\alpha}{N+2/\alpha}}\\ 
&\leqslant & \gamma 2^{j\gamma} a^{-\gamma}
k^{\frac{(m^{-}-1)p}{N+2/\alpha}+\frac{(\lambda_{+}+m^{-})(N+p)}{N+2/\alpha}}\Big(\frac{k^{1-\lambda_{+}}}{\eta}+\sum\limits_{i=1}^N \frac{k^{\lambda_i -\lambda_{+}}}{r^{p_{i}}_i}\Big)^{\frac{N+p}{N+2/\alpha}}\\
& & \times |Q_j \cap ]u<k_j]|^{1+\frac{p}{N+2/\alpha}}.
\end{eqnarray*}
On the other hand
\begin{eqnarray*} 
\iint\limits_{Q_j}(v^{m^{-}}-k^{m^{-}}_j)_{-}^{\alpha p}\zeta^{p_{+}\alpha p}_j\,dxdt & \geq & \iint\limits_{Q_{j+1}}(v^{m^{-}}-k^{m^{-}}_j)_{-}^{\alpha p}\,dxdt \geq 
(k^{m^{-}}_j-k^{m^{-}}_{j+1})^{\alpha p} |Q_{j+1}\cap [ v < k_{j+1}]| \\
& \geq &
k^{m^{-} \alpha p} 
\left(\frac{1-a^{m^{-}}}{2^{j+1}}\right)^{\alpha p} |Q_{j+1}\cap [v< k_{j+1}]|. 
\end{eqnarray*}
Gathering the previous estimates and setting $y_j:=\dfrac{|Q_j\cap [ v<k_j]|}{|Q_j|}$, we get
\[
y_{j+1}\leqslant 
\gamma 2^{j\gamma} a^{-\gamma}\Big(\frac{\eta}{k^{1-\lambda_{+}}}\Big)^{\frac{p}{N+2/\alpha}}\Big(\prod\limits_{i=1}^N \frac{r_i}{k^{\frac{\lambda_i-\lambda_{+}}{p_i}}}\Big)^{\frac{p}{N+2/\alpha}}\Big(\frac{k^{1-\lambda_{+}}}{\eta}+\sum\limits_{i=1}^N \frac{k^{\lambda_i-\lambda_{+}}}{r^{p_{i}}_i}\Big)^{\frac{N+p}{N+2/\alpha}}y_{j}^{1+\frac{p}{N+2/\alpha}}.\]
We now take $\nu_-$ such that
\begin{equation}\label{eq2.8}
\nu_-:=\frac{a^{\gamma}}{\gamma}\frac{k^{1-\lambda_{+}}}{\eta}\Bigg(\prod\limits_{i=1}^N \frac{r_i}{k^{\frac{\lambda_i-\lambda_{+}}{p_i}}}\Bigg)^{-1}\Big(\frac{k^{1-\lambda_{+}}}{\eta}+\sum\limits_{i=1}^N \frac{k^{\lambda_i-\lambda_{+}}}{r^{p_{i}}_i}\Big)^{-\frac{N+p}{p}}.
\end{equation}
and then, recalling de definition of $v$, one has  $y_0\leqslant \nu_-$ and thereby $\lim\limits_{j\rightarrow \infty}y_j=0$,
which proves \eqref{eq2.5}.
\end{proof}

\noindent We now present and prove an analogous result for locally bounded sub-solutions. Denote by $\mu^{\pm}$, $\omega$ the numbers satisfying 
\begin{equation}\label{boundsubsol}
    \mu^+ = \esssup\limits_{\Omega_T} u, \qquad \mu^-=\essinf\limits_{\Omega_T} u,\qquad \omega =  \mu^+-\mu^-.
\end{equation} 

\begin{lemma}\label{DGsubsol}
Let $u$ be a non-negative, locally bounded, local weak sub-solution to \eqref{eq1.1} under conditions \eqref{eq1.2}. 
There exists $\nu_+ \in (0, 1)$, depending only on the data, $a$, $\xi$, $\vec{r}$, $\eta$ and $\omega$, such that if
\begin{equation}\label{eq2.6}
|Q_{\vec{r},\eta}(y, \tau)\cap\{u^{m^{-}}\geqslant [\mu^{+}]^{m^{-}}-[\xi \omega]^{m^{-}}\}|\leqslant \nu_+ |Q_{\vec{r},\eta}(y, \tau)|,
\end{equation}
then
\begin{equation}\label{eq2.7}
u^{m^{-}}\leqslant [\mu^{+}]^{m^{-}}-[a\xi \omega]^{m^{-}},\quad (x,t)\in Q_{\vec{r}/2, \eta/2}(y, \tau).
\end{equation}
\end{lemma}

\begin{proof}
\noindent To prove \eqref{eq2.7} we consider levels $k^{m^{-}}_j=[\mu^{+}]^{m^{-}}-[a \xi \omega]^{m^{-}}-(1-a^{m^{-}})[\xi \omega]^{m^{-}}2^{-j}$ and take the same  cubes $K_j$, cylinders $Q_j$ and cutoff functions $\zeta_j$ as in the proof of Lemma \ref{lem2.4}. By Lemma \ref{lem2.3}
\[ \sup\limits_{t\in I_j}\int\limits_{K_j}g_+(u^{m^-}, k_j^{m^-})\,\zeta^{p_{+}}_j \,dx  + \gamma^{-1}\sum\limits_{i=1}^N \iint\limits_{Q_j}u^{(m_i-m^{-})(p_i-1)}\left|\left((u^{m^{-}}-k^{m^{-}}_j)_{+}\right)_{x_i}\right|^{p_{i}}\zeta^{p_{+}}_j\,dx\,dt\]
\[
\leqslant  \frac{\gamma 2^{j }}{\eta}\iint\limits_{Q_j}g_+(u^{m^-}, k_j^{m^-})\,\,dx dt+
\gamma 2^{jp^+}\sum\limits_{i=1}^{N}r_{i}^{-p_{i}}\iint\limits_{Q_j}u^{(m_i-m^-)(p_1-1)} (u^{m^-}-k_j^{m^-})_+^{p_i}\,dx\,dt \ . \]
Since
\[\frac{1}{2}\Big([\mu^{+}]^{m^{-}}-[\xi \omega]^{m^{-}}\Big)^{\frac{1-m^{-}}{m^{-}}}  (u^{m^{-}}-k^{m^{-}}_j)^2_{+} \leq g_+(u^{m^-}, k_j^{m^-}) \leq \frac{1}{2} [\mu^{+}]^{1-m^{-}} (\xi \omega)^{2m^-}
\]
\[u^{(m_i-m^-)(p_1-1)} (u^{m^-}-k_j^{m^-})_+^{p_i} \leq [\mu^+]^{(m_i-m^-)(p_i-1)} [\xi \omega]^{m^- p_i}\]
and
\begin{equation*}
\iint\limits_{Q_j}u^{(m_i-m^{-})(p_i-1)}\left|\left((u^{m^{-}}-k^{m^{-}}_j)_{+}\right)_{x_i}\right|^{p_{i}}\zeta^{p_{+}}_j\,dx\,dt\geqslant \alpha_i^{-p_i}
\iint\limits_{Q_j}\left|\left((u^{m^{-}}-k^{m^{-}}_j)^{\alpha_i}_{+}\right)_{x_i}\right|^{p_{i}}\zeta^{p_{+}}_j\,dx\,dt,
\end{equation*}
from the previous estimates we arrive at
\[
\Big([\mu^{+}]^{m^{-}}-[\xi \omega]^{m^{-}}\Big)^{\frac{1-m^{-}}{m^{-}}} \sup\limits_{t\in I_j}\int\limits_{K_j\times\{t\}}(u^{m^{-}}-k^{m^{-}}_j)^2_{+}\zeta^{p_{+}}_j\,dx+ \sum\limits_{i=1}^{N} \iint\limits_{Q_j}\left|\left((u^{m^{-}}-k^{m^{-}}_j)^{\alpha_i}_{+}\right)_{x_i}\right|^{p_{i}}\zeta^{p_{+}}_j\,dx\,dt\]
\begin{eqnarray*}
&\leqslant & \gamma 2^{j\gamma}\Big([\mu^{+}]^{1-m^{-}}\frac{[\xi \omega]^{2m^{-}}}{\eta}+
\sum\limits_{i=1}^N [\mu^{+}]^{(m_i-m^{-})(p_i-1)}\frac{[\xi \omega]^{m^{-}p_i}}{r^{p_i}}\Big)|Q_j\cap [u> k_j]|\\
& = & 
\gamma 2^{j\gamma}[\mu^{+}]^{1-m^{-}}[\xi \omega]^{2m^{-}+\lambda_{+}-1}\Big(\frac{[\xi \omega]^{1-\lambda_{+}}}{\eta}+
\sum\limits_{i=1}^N \frac{[\xi \omega]^{\lambda_i-\lambda_{+}}}{r_{i}^{p_i}}\big[\frac{\mu^{+}}{\xi \omega}\big]^{\lambda_i-m^-(p_i-2)-1}\Big)|Q_j\cap\{u\geqslant k_j\}|\\
&=& \gamma 2^{j\gamma}[\mu^{+}]^{1-m^{-}}[\xi \omega]^{2m^{-}+\lambda_{+}-1} H_1(\xi\omega, \mu^{+}, \eta, \vec{r})\,\,|Q_j\cap [u>k_j]|.
\end{eqnarray*}
Arguing in a similar way as before, by means of H\"{o}lder's inequality and the Sobolev embedding \eqref{embedding} and considering $|A_j|=Q_{j}\cap [ u > k_{j}] $, we get 
\begin{eqnarray*}
(k^{m^{-}}_{j+1}-k^{m^{-}}_{j})^{\alpha p} |A_{j+1}| & \leqslant & 
\iint\limits_{Q_j}(u^{m^{-}}-k^{m^{-}}_j)_{+}^{\alpha p}\zeta^{p_{+}\alpha p}_j\,dxdt  \leq \gamma 2^{j\gamma}\Big([\mu^{+}]^{m^{-}}-[\xi \omega]^{m^{-}}\Big)^{-\frac{(1-m^{-})p}{m^{-}(N+2/\alpha)}}\\
& & \times \Big([\mu^{+}]^{1-m^{-}}[\xi \omega]^{2m^{-}+\lambda_{+}-1} H_1(\xi\omega, \mu^{+}, \eta, \vec{r})\Big)^{\frac{p+N}{N+2/\alpha}}\,\,|A_j|^{1+\frac{p}{N+2/\alpha}}.
\end{eqnarray*}
From this, setting $y_j:=\dfrac{|A_j|}{|Q_j|}$, we obtain
\begin{multline*}
y_{j+1}\leqslant \gamma 2^{j\gamma}\, a^{-\gamma}\,\Big(\frac{[\mu^{+}]^{m^{-}}}{[\mu^{+}]^{m^{-}}-[\xi\omega]^{m^{-}}}\Big)^{\frac{(1-m^{-})p}{m^{-}(N+2/\alpha)}} \Big(\frac{\mu^{+}}{\xi\omega}\Big)^{\frac{(1-m^-)N}{N+2/\alpha}}\Big(\frac{\eta}{[\xi \omega]^{1-\lambda_{+}}}\Big)^{\frac{p}{N+2/\alpha}}\Big(\prod_{i=1}^N \frac{r_i}{[\xi \omega]^{\frac{\lambda_i-\lambda_{+}}{p_i}}}\Big)^{\frac{p}{N+2/\alpha}}\\\times [H_1(\xi\omega, \mu^{+}, \eta, \vec{r})]^{\frac{p+N}{N+2/\alpha}}
y_j^{1+\frac{p}{N+2/\alpha}},
\end{multline*}
which yields $\lim\limits_{j\rightarrow \infty}y_j=0$, and thereby \eqref{eq2.7}, provided that $\nu_+$ is chosen to satisfy
\begin{equation}\label{eq2.9}
\nu_+=\frac{a^{\gamma}}{\gamma}\Big(\frac{[\mu^{+}]^{m^{-}}-[\xi\omega]^{m^{-}}}{[\mu^{+}]^{m^{-}}}\Big)^{\frac{1-m^{-}}{m^{-}}} \Big(\frac{\xi\omega}{\mu^{+}}\Big)^{\frac{(1-m^-)N}{p}} \frac{[\xi \omega]^{1-\lambda_{+}}}{\eta}\Bigg(\prod\limits_{i=1}^N \frac{r_i}{[\xi \omega]^{\frac{\lambda_i-\lambda_{+}}{p_i}}}\Bigg)^{-1}
[H_1(\xi\omega, \mu^{+}, \eta, \vec{r})]^{-\frac{p+N}{p}}.
\end{equation}

\end{proof}

\vspace{.1cm}

\noindent {\bf A Variant of De Giorgi-Type Lemma Involving ''Initial Data''} - Assume that $u$ is a non-negative, local weak super-solution to \eqref{eq1.1}, \eqref{eq1.2} in $\Omega_T$ and assume additionally that, for a certain time level $\bar{t} \in (0,T)$, there holds
\begin{equation}\label{eq2.10}
u(x, \bar{t})\geqslant k_0 \quad \text{for}\quad x\in K_{\vec{r}}(y) \quad \mbox{and for some} \quad k_0>0 .
\end{equation}

\noindent By writing the energy estimates \eqref{eq2.3} over the cylinder $Q_j=K_j\times (\bar{t}, \bar{t}+\eta)$ for test functions $(u^{m^{-}}-k_j^{m^-})_{-} \zeta_j(x)^{p^+}$, being 
\[ k_j^{m^-}:=(ak_0)^{m^-}+(1-a^{m^-})k_0^{m^-}\,2^{-j} , \quad 0<a<1  , \]
\[ K_j:=K_{\vec{r}_j}(y)  \quad \
\mbox{where} \quad  \vec{r}_j:=\frac{\vec{r}}{2}(1+2^{-j})  \]
and $\zeta_j(x)$ a time independent cutoff function,  and reasoning analogously to what was done in the previous lemma when proving \eqref{eq2.5}, we obtain
\begin{equation*}
y_{j+1}\leqslant \gamma 2^{j\gamma} a^{-\gamma}\bigg(\frac{\eta}{k_0^{1-\lambda_{+}}}\bigg)^{\frac{p}{N+2/\alpha}}\Bigg(\prod\limits_{i=1}^N \frac{r_i}{k_0^{\frac{\lambda_i-\lambda_{+}}{p_i}}}\Bigg)^{\frac{p}{N+2/\alpha}}\Big(\sum\limits_{i=1}^N \frac{k_0^{\lambda_i-\lambda_{+}}}{r^{p_{i}}_i}\Big)^{\frac{N+p}{N+2/\alpha}}y_{j}^{1+\frac{p}{N+2/\alpha}},
\end{equation*}
where $y_j:=\dfrac{|Q_j\cap [v< k_j]|}{|Q_j|}$ and $v =\max \{u, a k_0\}$. Hence, provided that 
\begin{equation}\label{eq2.11}
\frac{| Q_{\vec{r}, \eta}^+ \cap [u<k_0]|}{| Q_{\vec{r}, \eta}^+ |}
\leqslant \frac{a^{\gamma}}{\gamma}\frac{k_0^{1-\lambda_{+}}}{\eta}\Bigg(\prod\limits_{i=1}^N \frac{r_i}{k_0^{\frac{\lambda_i-\lambda_{+}}{p_i}}}\Bigg)^{-1}\Big(\sum\limits_{i=1}^N \frac{k_0^{\lambda_i-\lambda_{+}}}{r^{p_{i}}_i}\Big)^{-\frac{N+p}{p}}:=\nu^{(1)}_-.
\end{equation}
we get $\lim\limits_{j\rightarrow \infty}y_j=0$. So, we summarize
\begin{lemma}\label{lem2.5}
Let $u$ be a non-negative, local weak super-solution to \eqref{eq1.1}-\eqref{eq1.2} and take $0<a<1$. Assume that $k_0>0$ is such that \eqref{eq2.10} and \eqref{eq2.11} hold, then
\begin{equation}\label{eq2.12}
u(x,t)\geqslant a\,k_0,\quad (x,t)\in K_{\frac{\vec{r}}{2}}(y)\times (\bar{t}, \bar{t}+\eta).
\end{equation}
\end{lemma}

\vspace{.2cm}

\noindent Let as now consider $u$ to be a non-negative, locally bounded, local weak sub-solution to \eqref{eq1.1}-\eqref{eq1.2} in $\Omega_T$. Consider the numbers $\mu^+$ and $\omega$ as in \eqref{boundsubsol}, and assume additionally that, for a certain time level $\bar{t} \in (0,T)$ and for some $\xi_0 \in (0,1)$,
\begin{equation}\label{eq2.13}
[u(x, \bar{t})]^{m^-}\leqslant [\mu^+]^{m^-}-[\xi_0 \omega]^{m^-} \quad \text{for}\quad x\in K_{\vec{r}}(y) \ . 
\end{equation}
 
\noindent By writing the energy inequalities \eqref{eq2.3} for $(u^{m^{-}}-k_j^{m^-})_{+} \zeta_j(x)^{p^+}$ over the cylinder $Q_j=K_j\times (\bar{t}, \bar{t} + \eta)$ (the same as before), where 
\[ k_j^{m^-}:=(\mu^{+})^{m^-} -(a\xi_0\omega)^{m^-}- (1-a^{m^-})(\xi_0\omega)^{m^-}\,2^{-j} , \quad 0<a<1  , \]
and $\zeta_j(x)$ is a time independent cutoff function and redoing  the arguments of the previous proof of \eqref{eq2.7}, we arrive at 
\begin{multline*}
y_{j+1}\leqslant \gamma 2^{j\gamma}\, a^{-\gamma}\,\Big(\frac{[\mu^{+}]^{m^{-}}}{[\mu^{+}]^{m^{-}}-[\xi_0\omega]^{m^{-}}}\Big)^{\frac{(1-m^{-})p}{m^{-}(N+2/\alpha)}}\Big(\frac{\eta}{[\xi_0 \omega]^{1-\lambda_{+}}}\Big)^{\frac{p}{N+2/\alpha}}\Big(\prod_{i=1}^N \frac{r_i}{[\xi_0 \omega]^{\frac{\lambda_i-\lambda_{+}}{p_i}}}\Big)^{\frac{p}{N+2/\alpha}}\\
\times  \left( \frac{\mu^+}{\xi_0 \omega}\right)^{(1-m^-)\frac{N}{N+2/\alpha}}\Big[\sum\limits_{i=1}^N \frac{[\xi_0 \omega]^{\lambda_i-\lambda_{+}}}{r_{i}^{p_i}}\big[\frac{\mu^{+}}{\xi_0 \omega}\big]^{\lambda_i-m^-(p_i-2)-1}\Big]^{\frac{p+N}{N+2/\alpha}} y_j^{1+\frac{p}{N+2/\alpha}}, 
\end{multline*}
for $y_j:=\dfrac{|Q_j\cap[ u> k_j]|}{|Q_j|}$. Once we choose 
\begin{multline*}
\nu^{(1)}_{+}:= 
\frac{a^{\gamma}}{\gamma}\Big(\frac{[\mu^{+}]^{m^{-}}-[\xi_0\omega]^{m^{-}}}{[\mu^{+}]^{m^{-}}}\Big)^{\frac{1-m^{-}}{m^{-}}}
\frac{[\xi_0 \omega]^{1-\lambda_{+}}}{\eta}\Big(\prod_{i=1}^N \frac{r_i}{[\xi \omega]^{\frac{\lambda_i-\lambda_{+}}{p_i}}}\Big)^{-1} \\\times \left( \frac{\xi_0 \omega}{\mu^+}\right)^{(1-m^-)\frac{N}{p}} \Big[\sum\limits_{i=1}^N \frac{[\xi_0 \omega]^{\lambda_i-\lambda_{+}}}{r_{i}^{p_i}}\big[\frac{\mu^{+}}{\xi_0 \omega}\big]^{\lambda_i-m^-(p_i-2)-1}\Big]^{-\frac{p+N}{p}}
\end{multline*}
we have \begin{equation}\label{eq2.14}
\frac{|Q^+_{\vec{r}, \eta}(y, \bar{t}) \cap [u> [\mu^{+}]^{m^{-}}-[\xi_0\omega]^{m^{-}}]|}{|Q^+_{\vec{r}, \eta}(y, \bar{t})|}\leqslant \nu^{(1)}_{+}
\end{equation} and therefore $\lim\limits_{j\rightarrow \infty}y_j=0$. We have just proved

\begin{lemma}\label{lem2.6}
Let $u$ be a non-negative, locally bounded, local weak sub-solution to \eqref{eq1.1}, \eqref{eq1.2}, and take $0<a<1$. Assume that $\xi_0$ is such that \eqref{eq2.13} and \eqref{eq2.14} hold, then
\begin{equation}\label{eq2.15}
[u(x,t)]^{m^-}\leqslant [\mu^+]^{m^-}-[a\,\xi_0 \omega]^{m^-},\quad (x,t)\in K_{\frac{\vec{r}}{2}}(y)\times (\bar{t}, \bar{t}+\eta).
\end{equation}
\end{lemma}

\section{Expansion of Positivity}\label{Sec.4}

\noindent In what follows we prove, separately for the degenerate and singular cases, results on the expansion of positivity - a property of non-negative super-solutions to (elliptic) and parabolic PDEs - stating that the information on the measure of the positivity set of $u$, located at a certain time level, over the cube $K_r^k(y)$ is expanded to a wider domain, both in space and time. These results are obtained assuming 
\begin{equation}\label{eq4.1}
\lambda_+-\lambda_-\leqslant \epsilon_*,
\end{equation}
for some $\epsilon_* \in (0,1)$ to be determined along the proofs depending only on the data.

\vspace{.2cm}

\noindent Through out this section we will consider $u$ to be a non-negative, local weak super-solution to \eqref{eq1.1}-\eqref{eq1.2} in $\Omega_T$.

\subsection{The Degenerate Case}

Assume that \eqref{eq4.1} holds, for $\epsilon_*$ to determined, and that $\lambda_+>1$, so we are considering the doubly nonlinear anisotropic PDE \eqref{eq1.1} to be degenerate. In this setting, the main result concerning the spreading of positivity of $u$ is the following
\begin{theorem}\label{th4.1}
Let $k, r >0$ and $\alpha_0 \in (0,1)$ and assume that
\begin{equation}\label{eq4.2}
|K^{k}_r(y)\cap [u(\cdot, s)\geqslant k]|\geqslant \alpha_0 |K^{k}_r(y)|.
\end{equation}
Then there exist constants $\epsilon_{*}\in (0,1)$, $C_{*}>1$ and $B_1>1$, depending only on the data and $\alpha_0$, such that
\begin{equation}\label{eq4.3}
u(x, t)\geqslant \frac{k}{C_{*}},\quad x\in K^{k/C_{*}}_{2r}(y),
\end{equation}
and for all $t$
\begin{equation}\label{eq4.4}
s+B_1\,r^{p_{+}}\,\,k^{1-\lambda_{+}}\leqslant t\leqslant s+\tilde{B_1} \,r^{p_{+}}\,\,k^{1-\lambda_{+}}.
\end{equation}
\end{theorem}

\noindent The proof of Theorem \ref{th4.1} relies on several other results. Roughly speaking, from \eqref{eq4.2} one shows that something similar happens in a larger cube, for all time levels from $s$ to $s+\eta$, where $\eta$ depends on $r, k$ and some other small numbers (Lemma \ref{lem4.1}). From that, one finds a level $k_*$ for which the portion of a cylinder where $u<k_*$ ($u$ is close to zero) can be made arbitrarily small (Lemma \ref{lem4.2}). By a DeGiorgi argument, that smallness is fixed, therefore a bound from below is obtained.

\vspace{.2cm}

\noindent Let \eqref{eq4.2} be in force. For the time being, consider $C_{*}>1$ fixed (it   will be specified later, depending only on the data and $\alpha_0$). By noting that $K^{k}_{r}(y) \subset K^{k/C_{*}}_{4 r}(y)$  and considering 
\begin{equation}\label{eq4.6}
\epsilon_{*}\leqslant \frac{p \ln 4}{N \ln C_{*}},
\end{equation}
one has 
\begin{eqnarray*}
     |K^{k}_{r}(y)|& =& \displaystyle{|K^{k/C_{*}}_{4 r}(y)| \ 4^{- p^+ N/p} \ C_{*}^{-\sum_{i=1}^N (\lambda_+-\lambda_i)/{p_i}}}\geq |K^{k/C_{*}}_{4 r}(y)| \ 4^{- p^+ N/p } \ C_{*}^{-\epsilon_{\star}N/p} \\
     & \geq &|  K^{k/C_{*}}_{4 r}(y)| \ 4^{- p^+ N/p -1 } \ , 
\end{eqnarray*}
therefore the measure information \eqref{eq4.2} yields
\begin{equation}\label{eq4.5}
|K^{k/C_{*}}_{4 r}(y)\cap [ u(\cdot, s)> k]|\geqslant \frac{\alpha_0}{4^{\frac{N p_{+}}{p} +1}} |K^{k/C_{\star}}_{4r}(y)|.
\end{equation}
\begin{lemma}\label{lem4.1}
Assume \eqref{eq4.5} holds. There exist $\delta$, $\epsilon \in (0,1)$, depending only on the data
and $\alpha_0$, such that for any $0< \tau \leqslant \tau_{*}\leqslant \ln C_{*}$ 
\begin{equation}\label{eq4.7}
|K^{k/C_{*}}_{4r}(y)\cap [ u(\cdot, t)> \epsilon k e^{-\tau}] |\geqslant \frac{\alpha^2_0}{16^{\frac{N p_{+}}{p}+2}} |K^{k/C_{*}}_{4r}(y)|, \qquad \forall s<t \leq s+\eta
\end{equation}
 where  $\eta:=\delta r^{p_{+}} \Big(\frac{e^{\tau}}{k}\Big)^{\lambda_{+}-1}.$
\end{lemma}
\begin{proof}
For fixed $\sigma \in (0, 1)$ set 
\[ \bar{K}^{k/C_*}_{4r}(y):=\{x: |x_i-y_i|\leqslant (1-\sigma) (4r)^{\frac{p_{+}}{p_i}} \Big(\dfrac{C_*}{k}\Big)^{\frac{\lambda_{+}- \lambda_i}{p_i}},\quad i=1,..., N\}  \subset K^{k/C_{*}}_{4r}(y) \ . \]
Take $\tau >0$ and consider the energy estimates \eqref{eq2.3} written over the cylinder $K^{k/C_{*}}_{4r}(y) \times (s, s+\eta]$ for the test functions $(u^{m^-}- (e^{-\tau} k)^{m^-} )_{-}\xi(x)^{p^+}$, where $\zeta(x)$ is a time-independent piecewise smooth cutoff function satisfying
\[ \quad 0 \leq \zeta \leq 1 \ \ \mbox{in} \ \ K^{k/C_{*}}_{4r}(y)  , \quad \zeta= 1 \ \  \mathrm{in} \ \  \bar{K}^{k/C_*}_{4r}(y) \quad \mbox{and} \quad |\zeta_{x_i}|\leqslant \dfrac{\sigma^{-1}}{(4r)^{\frac{p_{+}}{p_i}}}\Big(\dfrac{k}{C_*}\Big)^{\frac{\lambda_+-\lambda_i}{p_i}}, i=1, ..., N. \]
Hence \\
$ \displaystyle{\sup\limits_{s\leqslant t\leqslant s+\eta}\int\limits_{\bar{K}^{ k/C_{*}}_{4r}(y)} g_{-}(u^{m-}, (e^{-\tau}k)^{m^{-}}) \,\,dx \leq \int\limits_{K^{k/C_{*}}_{4r}(y)\times\{s\}} g_{-}(u^{m-}, (e^{-\tau}k)^{m^{-}}) \,\,dx }$
\[ + \frac{\gamma}{r^{p_{+}}\sigma^{p_{+}}} \sum\limits_{i=1}^{N}\big(k/C_{*}\big)^{\lambda_{+}-\lambda_i}\int\limits_s^{s+\eta}\!\!\int\limits_{K^{k/C_{*}}_{4r}(y)}u^{(m_i-m^{-})(p_i-1)}(u^{m^{-}}-(e^{-\tau} k)^{m^{-}})_{-}^{p_{i}}\,dx\,dt \ .\]
Observe that 
\begin{enumerate}
    \item[] $\displaystyle{g_{-}(u^{m-}, (e^{-\tau}k)^{m^{-}}) = m^{-} \int_{u}^{e^{-\tau}k} (s^{m^-} - (e^{-\tau}k)^{m^{-}})_- \ ds } $ is bounded from below, when considering $u<\epsilon k$, and above by
\[ \frac{(m^-)^2}{m^-+1} (e^{-\tau}k)^{m^{-}+1} \left(1- \frac{m^-+1}{m^-} \epsilon \right) \leq g_{-}(u^{m-}, (e^{-\tau}k)^{m^{-}}) \leq \frac{(m^-)^2}{m^-+1} (e^{-\tau}k)^{m^{-}+1} \ . \]
    \item[] $\displaystyle{\frac{\gamma}{r^{p_{+}}\sigma^{p_{+}}} \sum\limits_{i=1}^{N}\big(k/C_{*}\big)^{\lambda_{+}-\lambda_i}\int\limits_s^{s+\eta}\!\!\int\limits_{K^{k/C_{*}}_{4r}(y)}u^{(m_i-m^{-})(p_i-1)}(u^{m^{-}}-(e^{-\tau} k)^{m^{-}})_{-}^{p_{i}}\,dx\,dt } $
   \[ \leq \frac{\gamma}{\sigma^{p_{+}}} (e^{-\tau}k)^{m^{-}+1} | K^{k/C_{*}}_{4r}(y)| \sum\limits_{i=1}^{N}\big(e^\tau/C_{*}\big)^{\lambda_{+}-\lambda_i} \leq  \frac{\gamma}{\sigma^{p_{+}}} (e^{-\tau}k)^{m^{-}+1} | K^{k/C_{*}}_{4r}(y)|  
   \ , \quad  \ \tau \leq \tau_*\leq \ln C_*  \ . \]
\end{enumerate}
Gathering all theses estimates we arrive at 
\begin{equation*}
|K^{k/C_{*}}_{4r}(y)\cap\{u(\cdot, t)<\epsilon e^{-\tau} k\}|\leqslant \Big[N\sigma+ \frac{1-\frac{\alpha_0}{4^{\frac{Np_{+}}{p}+1}}}{1-\frac{1+m^{-}}{m^{-}}\epsilon}+\frac{\delta\,\gamma}{\sigma^{p_{+}}}\big(1-\frac{1+m^{-}}{m^{-}}\epsilon\big)^{-1}\Big]|K^{k/C_{*}}_{4r}(y)|.
\end{equation*}
The proof is complete once we choose $\epsilon \in (0,1)$ such that \[\dfrac{1}{1-\frac{1+m^{-}}{m^{-}}\epsilon}=1+\dfrac{\alpha_0}{4^{\frac{Np_{+}}{p}+1}} \]
and then choose $\delta$, $\sigma$ such that
$N\sigma+\frac{\delta\,\gamma}{\sigma^{p_{+}}}\left(1+\dfrac{\alpha_0}{4^{\frac{Np_{+}}{p}+1}}\right)=\dfrac{15 \alpha^2_0}{16^{\frac{Np_{+}}{p}+2}}$.
\end{proof}

\noindent The next step gives us information on the existence of a level $k_*$ such that the portion of a proper cylinder $Q_*$, where $u<k_*$, can be made arbitrarily small. For that purpose we introduce the new (unknown) function 
\begin{equation*}
w(z, \tau):=\frac{e^{\tau}}{k}\,u(y_1+z_1 \frac{r^{\frac{p_{+}}{p_1}}}{k^{\frac{\lambda_{+}-\lambda_1}{p_1}}},..., y_N+z_N \frac{r^{\frac{p_{+}}{p_N}}}{k^{\frac{\lambda_{+}-\lambda_N}{p_N}}}, s+ \delta\,\Big(\frac{e^{\tau}}{k}\Big)^{\lambda_{+}-1}r^{p_{+}})
\end{equation*}
which by Lemma \ref{lem4.1} satisfies 
\begin{equation}\label{eq4.8}
|K^{1/C_{*}}_{4}(0)\cap [w(\cdot, \tau)\geqslant \epsilon ] |\geqslant \frac{\alpha^2_0}{16^{\frac{Np_{+}}{p}+2}}|K^{1/C_{*}}_{4}(0)|,\quad\text{for all}\quad 0<\tau\leqslant \tau_{*}\leqslant \ln C_{*},
\end{equation}
and verifies the inequality
\begin{equation}\label{eq4.9}
w_\tau\geqslant \delta\,(\lambda_{+}-1)\Big(\frac{e^{\tau}}{k}\Big)^{\lambda_{+}}\,r^{p_{+}}\, u_t=\delta (\lambda_{+}-1)\sum\limits_{i=1}^N\,
\big(\bar{a}_i(\tau, z, w, D w)\big)_{z_i}
\end{equation}
in $Q:=B^{1/C_{*}}_4(0)\times(0, \tau_{*}]$, where $\tau_{*}>0$ is to be defined, and $\bar{a}_i(\tau, z, w, D w)$ satisfy the conditions
\begin{equation} \label{newstructcond}
\begin{cases}
\sum\limits_{i=1}^N \bar{a}_i(\tau, z, w, D w) w_{z_i}\geqslant K_1 \sum\limits_{i=1}^N e^{\tau(\lambda_{+}-\lambda_i)} w^{(m_i-1)(p_i-1)}|w_{z_i}|^{p_i},\\[.8em]
|\bar{a}_i(\tau, z, w, D w)|\leqslant K_2\,e^{\tau(\lambda_{+}-\lambda_i)}\left(\sum\limits_{j=1}^N w^{(m_i-1)+(m_j-1)(p_j-1)}|w_{z_j}|^{p_j}\right)^{\frac{pi}{p_i-1}}.
\end{cases}
\end{equation}



\noindent Under the hypothesis of Lemma \ref{lem4.1}, the function $w$ defined above satisfies \eqref{eq4.8} and verifies \eqref{eq4.9}-\eqref{newstructcond}. In this (new) setting one has

\begin{lemma}\label{lem4.2}
For any $\nu\in(0,1)$, there exists a positive number $j_{*}$, depending on the data, $\alpha_0$ and $\nu$, such that
\begin{equation}\label{eq4.10}
|Q_{*}\cap\{w\leqslant \frac{\epsilon}{2^{j_{*}}}\}|\leqslant \nu |Q_{*}|,\quad Q_{*}:=K^{1/C_{*}}_{4}(0)\times\left[\frac{\tau_{*}}{2},\tau_{*}\right],\quad \tau_{*}:=\Big(\frac{2^{j_{*}}}{\epsilon}\Big)^{\lambda_{+}-1},
\end{equation}
provided that
\begin{equation}\label{eq4.11}
\tau^{\frac{1}{\lambda_{+}-1}}_{*}\,\,e^{\tau_{*}}\leqslant C_{*}  \quad\text{and}\quad \epsilon_* \leqslant \frac{p_{-}}{\ln C_{*}}.
\end{equation}
\end{lemma}
\begin{proof}
Set $k_j:=\dfrac{\epsilon}{2^j}$ and \[ A_j(\tau):=K^{1/C_{*}}_{4}(0)\cap [w(\cdot, \tau)<k_j] \qquad \mbox{and}  \qquad |A_j|:=\int\limits_{\tau_{*}/2}^{\tau_{*}} |A_j(\tau)|\,d\tau \] for $j=0, \cdots, j_*-1$. The Discrete Isoperimetric Inequality (see \cite[Prop.5.1,Chap.10]{DB-PDEs}) together with \eqref{eq4.8} allow us to derive, for all $0<\tau\leqslant \tau_{*}$ 
\begin{eqnarray*}
k_{j+1}|A_{j+1}(\tau)|& \leqslant & \gamma(\alpha_0) \sum\limits_{i=1}^N C_*^{\frac{\lambda_+-\lambda_i}{p_i}} \int\limits_{A_j(\tau)\setminus A_{j+1}(\tau)}|w_{z_i}|\,dz \\
&\leqslant &\gamma(\alpha_0)  \sum\limits_{i=1}^N C_*^{\frac{\lambda_+-\lambda_i}{p_i}} \ k_{j}^{-\frac{(m_i-1)(p_i-1)+m^{-}-1}{p_i}}  \int\limits_{A_j(\tau)\setminus A_{j+1}(\tau)}w^{\frac{(m_i-1)(p_i-1)+m^{-}-1}{p_i}}|w_{z_i}|\,dz \\
&\leqslant & \gamma(\alpha_0)  \sum\limits_{i=1}^N C_*^{\frac{\lambda_+-\lambda_i}{p_i}} \ k_{j}^{-\frac{(m_i-1)(p_i-1)+m^{-}-1}{p_i}}\,\int\limits_{A_j(\tau)\setminus A_{j+1}(\tau)}w^{\frac{(m_i-m^-)(p_i-1)}{p_i}}\left| \left((w^{m^-} -k_j^{m-})_-\right)_{z_i} \right|\,dz.
\end{eqnarray*}
We now integrate in time over $\left[ \tau_*/2, \tau_*\right]$ and use H\"older's inequality to get, from the previous estimate,
\begin{multline} \label{Ajstar}
k_{j+1} \ |A_{j_{*}}|\leqslant \gamma(\alpha_0)  \sum\limits_{i=1}^N  C_*^{\frac{\lambda_+-\lambda_i}{p_i}} \  k_{j}^{-\frac{(m_i-1)(p_i-1)+m^{-}-1}{p_i}}\times\\\times\Big(\iint\limits_{Q_{*}}
w^{(m_i-m^{-})(p_i-1)}\left| \left((w^{m^-} -k_j^{m-})_-\right)_{z_i} \right|^{p_i}\,dz d\tau\Big)^{\frac{1}{p_{i}}} |A_j\setminus A_{j+1}|^{1-\frac{1}{p_{i}}},
\end{multline}
where $Q_*= K^{1/C_{*}}_4(0)\times [\tau_*/2, \tau_{*}] \subset K^{1/C_{*}}_8(0)\times [\tau_{*}/4, \tau_{*}]=Q_{*}'$.
By writing energy estimates over $Q_{*}'$ for $(w^{m^-} -k_j^{m^-})_- \zeta^{p^+}$, being $0 \leq \zeta\leq 1$ a piecewise smooth cutoff function defined in $Q_{*}'$ satisfying  $\zeta=1$ in $Q_{*}$, $|\zeta_\tau|\leqslant \gamma/\tau_{*}$ and $|\zeta_{z_i}|^{p_i}\leqslant \gamma \ C_{*}^{-(\lambda_{+}-\lambda_i)}$, $i=1,...,N$, one arrives at

$ \displaystyle{ \sum\limits_{i=1}^N \iint\limits_{Q_{*}}w^{(m_i-m^{-})(p_i-1)}\left| \left((w^{m^-} -k_j^{m-})_-\right)_{z_i} \right|^{p_i}\,dz d\tau }$
\vspace{-.2cm}
\begin{eqnarray*}
& \leqslant & 
\sum\limits_{i=1}^N \iint\limits_{Q_{*}} e^{\tau(\lambda_{+}-\lambda_i)}w^{(m_i-m^{-})(p_i-1)}\left| \left((w^{m^-} -k_j^{m-})_-\right)_{z_i} \right|^{p_i}\,dz d\tau \\ \nonumber
& \leqslant & \gamma(\delta)\left( k^{\lambda_{+}-1}_{j_{*}}k^{1+m^{-}}_j+\sum\limits_{i=1}^N \Big(\frac{e^{\tau_{*}}}{C_{*}}\Big)^{\lambda_{+}-\lambda_i}k_j^{\lambda_i+m^{-}}\right)|Q_{*}|
\\ \nonumber
&\leqslant &\gamma(\delta) k_j^{\lambda_{+}+m^{-}}\left(1+\sum\limits_{i=1}^N\Big(\frac{\tau^{\frac{1}{\lambda_{+}-1}}_{*} e^{\tau_{*}}}{C_{*}}\Big)^{\lambda_{+}-\lambda_i} \right)|Q_{*}|\leqslant \gamma(\delta)k_j^{\lambda_{+}+m^{-}}|Q_{*}|, 
\end{eqnarray*}
provided that \eqref{eq4.11} holds. Therefore, this last estimate together with \eqref{Ajstar} yield
\begin{equation*}
\frac{|A_{j_{*}}|}{|Q_{*}|}\leqslant \gamma(\delta, \alpha_0)\sum\limits_{i=1}^N (C_* k_j)^{\frac{\lambda_{+}-\lambda_{i}}{p_i}}\Big(\frac{|A_j\setminus A_{j+1}|}{|Q_{*}|}\Big)^{1-\frac{1}{p_i}}\leqslant\gamma(\delta, \alpha_0)\Big(\frac{|A_j\setminus A_{j+1}|}{|Q_{*}|}\Big)^{1-\frac{1}{p_{-}}}.
\end{equation*}
 Summing up the last inequality for  $0\leqslant j \leqslant j_{*}-1$ and choosing $j_*$ by the condition
\begin{equation*}
\frac{\gamma(\delta, \alpha_0)}{j_{*}^{1-\frac{1}{p_{-}}}}\leqslant \nu,
\end{equation*}
we arrive at the required \eqref{eq4.10}. Observe that from this choice, $j_*$ depends on the data, $\alpha_0$ and $\nu$.
\end{proof}

\vspace{.3cm}

\noindent The next and final step is to fix $\nu$ (and consequently one obtains $j_*$), depending only on the data, and then by arguing {\it a la De Giorgi} get an inferior bound to $u$. 

\noindent Observe that, when working over the bounded cylinder $Q_*$, $w$ is a bounded function that satisfies \eqref{eq4.9}-\eqref{newstructcond}, hence one can obtain a similar result to Lemma \ref{lem2.4} for $w$. In fact, from energy estimates written for $(w^{m^-} -k_j^{m^-})_- \zeta^{p^+}$ over $Q_j$, being $0 \leq \zeta\leq 1$ a piecewise smooth cutoff function defined in $Q_j$ and vanishing on its parabolic boundary, where 
\[K^{1/C_{*}}_{2}(0)\times\left[\frac{3\tau_{*}}{4},\tau_{*}\right]=: Q_{**} \subset Q_j \subset Q_*, \quad  \mbox{and} \quad \dfrac{k_{j_*}}{2} < k_j \leq k_{j_*} \]
we get 

$\displaystyle{
\sup\limits_{\tau}\int\limits_{K_j \times\{\tau\}}g_{-}(w^{m^{-}}, k_j^{m^{-}})\,\zeta^{p_{+}}\,dz+\gamma^{-1}\sum\limits_{i=1}^N \iint\limits_{Q_j}w^{(m_i-m^{-})(p_i-1)} \left| \left((w^{m^-} -k_j^{m-})_-\right)_{z_i} \right|^{p_{i}}\,\zeta^{p_{+}}dz\,d\tau}$
\begin{eqnarray*}
    & \leqslant & \frac{\gamma 2^{j}}{\tau_*}\iint\limits_{Q_j} g_{-}(w^{m^{-}}, k_j^{m^{-}})\,\,dz d\tau +\gamma \sum\limits_{i=1}^{N} 2^{j p_i} \ \frac{e^{\tau_*(\lambda_+-\lambda_i)}}{C_*^{\lambda_+-\lambda_i}} \iint\limits_{Q_j}w^{(m_i-m^{-})(p_i-1)}(w^{m^{-}}-k^{m^{-}}_j)_{-}^{p_{i}}\,dz\,d\tau \\
    & \leqslant & \gamma 2^{j p_+}k_{j_*}^{\lambda_+ +m^-} \left\{ \frac{k_{j_*}^{1-\lambda_+}}{\tau_*} + \sum\limits_{i=1}^{N} \left(\frac{e^{\tau_*}}{k_{j_*} C_*} \right)^{\lambda_+-\lambda_i} \right\} |Q_j \cap[w<k_j]| \leq \gamma 2^{j p_+} k_{j_*}^{\lambda_+ +m^-} |Q_j \cap[w<k_j]| \ .
\end{eqnarray*} 
These estimates were derived recalling that $\tau_{*}:=\Big(\frac{2^{j_{*}}}{\epsilon}\Big)^{\lambda_{+}-1}= k_{j_*}^{1-\lambda_*}$, $g_{-}(w^{m^{-}}, k_j^{m^{-}}) \leq k_{j_*}^{m^{-}+1} \chi_{[w<k_j]}$
$ 1 < e^{\tau(\lambda_+-\lambda_i)} \leqslant e^{\tau_*(\lambda_+-\lambda_i)}$ and assuming \eqref{eq4.11} holds. At this point we have an estimate like \eqref{est}, thereby we can proceed as in Lemma \ref{lem2.4} and fix $\nu$, depending only upon the data, for which \eqref{eq4.11} holds and conclude 
\begin{equation*}
w(z, \tau)\geqslant \frac{k_{j_{*}}}{2},\quad (z, \tau) \in K^{1/C_{*}}_2(0)\times [3\tau_{*}/4, \tau_{*}].
\end{equation*}
Returning to the original variables
\begin{equation}\label{eq4.12}
u(x, t)\geqslant \frac{\epsilon k}{2^{j_{*}+1} e^{\tau_{*}}},\quad (x,t)\in K^{k/C_{*}}_{2r}(y)\times \left(s+ \delta r^{p_{+}} \Big(\frac{e^{3\tau_{*}/4}}{k}\Big)^{\lambda_{+}-1}, s+ \delta r^{p_{+}} \Big(\frac{e^{\tau_{*}}}{k}\Big)^{\lambda_{+}-1}\right) \ .
\end{equation}

\noindent By taking 
\[ C_{*}=\dfrac{2^{1+j_*} e^{\tau_{*}}}{\epsilon},  \quad 
B_1=\delta  e^{3\tau_{*}(\lambda_{+}-1)/4}<\tilde{B_1}= \delta  e^{\tau_{*}(\lambda_{+}-1)}, \quad \mbox{and} \quad
\epsilon_*=\dfrac{p_{-}}{N \ln C_*} \]
 assumptions \eqref{eq4.6} and \eqref{eq4.11} are satisfied so, from \eqref{eq4.12}, we arrive at the required \eqref{eq4.3}. This completes the proof of Theorem \ref{th4.1}.

\subsection{The Singular Case}\label{Sec.5}

In what follows we assume, as before, that \eqref{eq4.1} holds for some sufficiently small positive number $\epsilon_*$ to be chosen; but now we will be considering \eqref{eq1.1} to be singular, meaning that $\lambda_+<1$. In this singular setting, the result on the expansion of positivity reads like 
\begin{theorem}\label{th5.1}
Let $k$ and $r$ be two positive numbers and $\alpha_0 \in (0,1)$. If
\begin{equation}\label{eq5.1}
|K^{k}_r(y)\cap\{u(\cdot, s)\geqslant k\}|\geqslant \alpha_0 |K^{k}_r(y)|,
\end{equation}
then there exist constants $\delta$, $\varepsilon$, $\epsilon_{*}\in(0,1)$ and $C_{*}>1$, depending only on the data and $\alpha_0$, such that
\begin{equation}\label{eq5.2}
u(x, t)\geqslant \frac{k}{C_{*}},\quad x\in K^{k/C_{*}}_{r}(y),
\end{equation}
and for all $t$
\begin{equation}\label{eq5.3}
s+(1-\varepsilon) \delta \,r^{p_{+}}\,\,k^{1-\lambda_{+}} \leqslant t\leqslant s+ \delta \,r^{p_{+}}\,\,k^{1-\lambda_{+}}.
\end{equation}

\end{theorem}

\vspace{.2cm}

\noindent Once we consider $\epsilon_*$ to verify \eqref{eq4.6}
from \eqref{eq5.1} we obtain 
\begin{equation}\label{eq5.5}
|K^{k/C_{*}}_{4r}(y)\cap \{ u(\cdot, s)> k\}|\geqslant |K^{k}_r(y)\cap\{u(\cdot, s) >k\}|\geqslant \alpha_0 |K^{k}_{r}(y)|
\geqslant  \frac{\alpha_0}{4^{\frac{Np_{+}}{p}+1}} |K^{k/C_{*}}_{4r}(y)|.
\end{equation}

\noindent Proceeding very closely to what was done in Lemma \ref{lem4.1}, we present (and prove analogously) the first result towards expansion on positivity.
\begin{lemma}\label{lem5.1}
There exist $\delta$, $\epsilon \in (0,1)$, depending only on the data
and $\alpha_0$, such that  
\begin{equation}\label{eq5.6}
|K^{k/C_*}_{4r}(y)\cap \{ u(\cdot, t)> \epsilon \,k\}|\geqslant \frac{\alpha^2_0}{16^{\frac{Np_{+}}{p}+2}} |K^{k/C_*}_{4r}(y)|,\quad\text{for all}\quad t\in(s, s+\delta r^{p_{+}}\,k^{1-\lambda_{+}}].
\end{equation}
\end{lemma}


\vspace{.2cm}

\noindent Introduce the new unknown function, defined in a new space-time configuration,
\begin{equation*}
w(z, \tau):=\frac{e^{\tau}}{k}\,u(y_1+z_1 \frac{r^{\frac{p_{+}}{p_1}}}{k^{\frac{\lambda_{+}-\lambda_1}{p_1}}},..., y_N+z_N \frac{r^{\frac{p_{+}}{p_N}}}{k^{\frac{\lambda_{+}-\lambda_N}{p_N}}}, s+ \delta k^{1-\lambda_{+}}\,r^{p_{+}}\,(1-e^{-\tau (1-\lambda_{+})}))
\end{equation*}
which due to \eqref{eq5.6} satisfies 
\begin{equation}\label{eq5.7}
|K^{1/C_*}_{4}(0)\cap\{w(\cdot, \tau)\geqslant \epsilon\,e^{\tau} \}|\geqslant \frac{\alpha^2_0}{16^{\frac{Np_{+}}{p}+2}}|K^{1/C_*}_{4}(0)|,\quad\text{for all}\quad \tau >0
\end{equation}
and 
\begin{equation}\label{eq5.8}
w_\tau\geqslant \delta (1-\lambda_{+}) \; \left(\frac{e^{\tau}}{k}\right)^{\lambda_+}\,r^{p_{+}}\, u_t=\delta (1-\lambda_{+})\sum\limits_{i=1}^N \big(\bar{a}_i(\tau, z, w, \nabla w)\big)_{z_i}
\end{equation}
in $Q:=K^{1/C_*}_4(0)\times(0, \infty)$ and $\bar{a}_i(\tau, z, w, \nabla w)$ satisfy the conditions
\begin{equation}\label{Nstrcond}
\begin{cases}
\sum\limits_{i=1}^N \bar{a}_i(\tau, z, w, D w) w_{z_i} \geqslant K_1 \sum\limits_{i=1}^N e^{\tau(\lambda_{+}-\lambda_i)} w^{(m_i-1)(p_i-1)}| w_{z_i}|^{p_i},\\[.8em]
|\bar{a}_i(\tau, z, w, D w)|\leqslant K_2\,e^{\tau(\lambda_{+}-\lambda_i)} \left( \sum\limits_{j=1}^N w^{m_i-1+(m_j-1)(p_j-1)}|w_{z_j}|^{p_j}\right)^{\frac{p_i-1}{p_i}}.
\end{cases}
\end{equation}
Let $\tau_0$ be a fixed positive number (to be chosen) and consider the levels
$$k_0=\epsilon e^{\tau_0}\quad \text{and}\quad k_j=\frac{k_0}{2^j},\quad j=0, 1, 2,...,j_{*}-1,$$
where $j_{*}$ is to be specified later. Then \eqref{eq5.7} can be rewritten as
\begin{equation}\label{eq5.9}
|K^{1/C_*}_{4}(0)\cap\{w(\cdot, \tau)\geqslant k_j\}|\geqslant \frac{\alpha^2_0}{16^{\frac{Np_{+}}{p}+2}}|K^{1/C_*}_{4}(0)|,\quad\text{for all}\quad \tau \geqslant \tau_0.
\end{equation}


\noindent Consider energy estimates written for $(w^{m^{-}}- k_j^{m^{-}})_{-} \zeta^{p^+}$ over $Q^{''}:=K^{1/C_*}_{8}(0)\times(\tau_0, \tau_0+2k_{0}^{1-\lambda_{+}})$, where $\zeta=\zeta(x,t) $ is a piecewise smooth cutoff function trapped between $Q^{'}:=K^{1/C_*}_{4}(0)\times (\tau_0 +k_{0}^{1-\lambda_{+}}, \tau_0+2k_{0}^{1-\lambda_{+}})$ and $Q^{''}$. We then get

\[\hspace{-5cm}\sum\limits_{i=1}^{N} \iint\limits_{Q^{'}} w^{(m_i-m^{-})(p_i-1)}\left|\left((w^{m^{-}}-k^{m^{-}}_j)_{-}\right)_{z_i}\right|^{p_{i}}\,dz\,d\tau \]
\vspace{-.5cm}
\begin{eqnarray*}
& \leqslant &\gamma(\delta)\, k_0^{\lambda_{+}-1}\iint\limits_{Q^{''}} g_{-}(w^{m^-}, k_j^{m^-}) \,dz d\tau\\ 
& & +\gamma(\delta) \sum\limits_{i=1}^{N}\Big(\frac{e^{\tau_0+2k^{1-\lambda_{+}}_0}}{C_*}\Big)^{\lambda_{+}-\lambda_i}\iint\limits_{Q^{''}}w^{(m_i-m^{-})(p_i-1)}(w^{m^{-}}-k^{m^{-}}_j)_{-}^{p_{i}}\,dz\,d\tau\\ 
& \leqslant & \gamma(\delta) \Big(k_{0}^{\lambda_{+}-1}k_j^{1+m^{-}}+\sum\limits_{i=1}^n \Big(\frac{e^{\tau_0+2k^{1-\lambda_{+}}_0}}{C_*}\Big)^{\lambda_{+}-\lambda_i} k_j^{\lambda_i+m^{-}}\Big)|Q^{'}| \\ 
& \leqslant &
\gamma(\delta)  k_j^{\lambda_{+}+m^{-}} \left( 1+ \sum\limits_{i=1}^n \Big(\frac{e^{\tau_0+2k^{1-\lambda_{+}}_0}}{C_* \, k_j}\Big)^{\lambda_{+}-\lambda_i} \right)|Q^{'}| 
\end{eqnarray*}
therefore
\begin{equation} \label{eq5.10}
\sum\limits_{i=1}^{N} \iint\limits_{Q^{'}} w^{(m_i-m^{-})(p_i-1)}\left|\left((w^{m^{-}}-k^{m^{-}}_j)_{-}\right)_{z_i}\right|^{p_{i}} \,dz\,d\tau \leqslant \gamma(\delta, \epsilon) \;  k_j^{\lambda_{+}+m^{-}} |Q^{'}|  
\end{equation}
provided that
\begin{equation}\label{eq5.11}
e^{\tau_0} \geq 2^{j_*} \qquad \mbox{and} \qquad  e^{\tau_0 + 2k^{1-\lambda_{+}}_0}\leqslant C_* \ .
\end{equation}

\noindent Now we use the isoperimetric lemma (a De Giorgi-Poincaré type Lemma) together with \eqref{eq5.9} to obtain, for 
$A_j (\tau)= K_{4}^{1/C_*}(0)\cap [w(\cdot,\tau)< k_j]$, $\tau_0 + k_0^{1-\lambda_+} \leqslant \tau \leqslant \tau_0 + 2 k_0^{1-\lambda_+}$ and $j=0, \cdots, j_*-1$,
\begin{eqnarray*}
k_{j+1}|A_{j+1}(\tau)| &\leqslant &\gamma(\alpha_0)\sum\limits_{i=1}^N C_*^{\frac{\lambda+-\lambda_i}{p_i}} \int\limits_{K_{4}^{1/C_*}(0)\cap [k_{j+1}< w(\cdot, \tau)< k_j]}| w_{z_i}| \; dz \\
& \leqslant & \gamma(\alpha_0)\sum\limits_{i=1}^N C_*^{\frac{\lambda+-\lambda_i}{p_i}}\ k_{j}^{-\frac{(m_i-1)(p_i-1)+m^{-}-1}{p_i}}\int\limits_{K_{4}^{1/C_*}(0)\cap [k_{j+1}< w(\cdot, \tau)< k_j ] } w^{\frac{(m_i-1)(p_i-1)+m^{-}-1}{p_i}}| w_{z_i}| dz \\
& \leqslant & \gamma(\alpha_0)\sum\limits_{i=1}^N C_*^{\frac{\lambda+-\lambda_i}{p_i}}\ k_{j}^{-\frac{(m_i-1)(p_i-1)+m^{-}-1}{p_i}}\int\limits_{A_j (\tau)\setminus A_{j+1} (\tau) } w^{(m_i-m^{-})(p_i-1)} \left|\left((w^{m^{-}}-k^{m^{-}}_j)_{-}\right)_{z_i}\right|^{p_{i}} \,dz  \ .
\end{eqnarray*}
Integrating the last inequality over the indicated time interval, using H\"{o}lder's inequality and \eqref{eq5.10} we obtain
\begin{equation*}
|A_{j_{*}}|\leqslant\gamma(\alpha_0, \delta, \epsilon)\sum\limits_{i=1}^N \left( C_* k_{j}\right)^{\frac{\lambda_{+}-\lambda_i}{p_i}}|Q^{'}|^{\frac{1}{p_i}}|A_j\setminus A_{j+1}|^{1-\frac{1}{p_i}}\leqslant \gamma(\alpha_0, \delta) \sum\limits_{i=1}^N |Q^{'}|^{\frac{1}{p_i}}|A_j\setminus A_{j+1}|^{1-\frac{1}{p_i}} \ .
\end{equation*}

\noindent Summing up the previous inequality for $0\leqslant j\leqslant j_*-1$, we obtain
\begin{equation}\label{eq5.13}
\frac{|A_{j_{*}}|}{|Q^{'}|}\leqslant \frac{\gamma(\alpha_0, \delta)}{j^{1-\frac{1}{p_{-}}}_{*}}.
\end{equation}

\noindent We have just proved that for $w$ satisfying \eqref{eq5.7}-\eqref{eq5.8}-\eqref{Nstrcond}, and under assumption \eqref{eq5.11},

For given an arbitrary positive small $\nu \in (0,1)$, there exists $j_\star$ (depending on the data and on $\alpha_o$), such that
\begin{equation}\label{eqnewsingcase}
\left|Q^{'} \cap \left[w < k_{j_*}\right] \right| \leq \nu \ |Q^{'} | \ .
\end{equation}

\vspace{.2cm} 

\noindent We are on the edge of proving Theorem \ref{th5.1}: consider the cylinder $Q^{'}=K^{1/C_*}_{4}(0)\times (\tau_0 +k_{0}^{1-\lambda_{+}}, \tau_0+2k_{0}^{1-\lambda_{+}}]$ partitioned into $2^{j_*(1-\lambda_+)}$ disjoint sub-cylinders $Q^i_{(j_{*})}$ of smaller length $k_{j_{*}}^{1-\lambda_{+}}$. Assuming \eqref{eqnewsingcase} holds true, for sure it also holds true at least in one of the sub-cylinders $Q^i_{(j_{*})}$
\[ \left|Q^i_{(j_{*})} \cap \left[w < k_{j_*}\right] \right| \leq \nu \ |Q^i_{(j_{*})} | \ . \]
Then, by arguing similarly to what was done in Lemma \ref{lem2.4}, taking $r_i=C^{\frac{\lambda_{+}-\lambda_i}{p_i}}_*$, $\eta=k_{j_{*}}^{1-\lambda_{+}}$, $k= k_{j_{*}}$ (and $a=1/2$), and assuming 
\begin{equation}\label{epsilonstarsing}
\epsilon_* \leq \frac{p_{-}}{N \ln C_*} \ , 
\end{equation}
one can find the number $\nu\in (0,1)$, depending only on the data (and independent of $\tau_0$, $k_{j_*}$ and $C_*$) and conclude that
\begin{equation*}
w(z, \tau_1)\geqslant \frac{k_{j_*}}{2},\quad z\in K^{1/C_*}_2(0),
\end{equation*}
and for some time level $\tau_0+k^{1-\lambda_{+}}_0\leqslant \tau_1 \leqslant \tau_0+ 2k^{1-\lambda_+}_0$. Returning to the original variables
\begin{equation}\label{eq5.14}
u(x, t_1)\geqslant \frac{\epsilon k}{2^{j_*+1} e^{\tau_1-\tau_0}}:= k_1,\quad x\in K^{k/C_*}_{2r}(y),
\end{equation}
for the time level $t_1=s+\delta k^{1-\lambda_{+}} r^{p_{+}} (1-e^{-\tau_1(1-\lambda_{+})})$. The idea now is to use De Giorgi type Lemma \ref{lem2.5} to conclude that something alike to \eqref{eq5.14} is also true in a full cylinder of the type $\tilde{Q} =K^{k/C_*}_{r}(y) \times (t_1, t_1 + \nu_o k_1^{1-\lambda_+} r^{p^+}] $, for some $\nu_o \in (0,1)$.

\noindent Observe that when proceeding as in the proof of Lemma \ref{lem2.5} taking $k_o=k_1$, $r_i= \dfrac{r^{p^+/p_i}}{\left(k/C_*\right)^{\lambda_+-\lambda_i}}$, $a=1/2$, and assuming  \eqref{eq4.1}, \eqref{eq4.6} and \eqref{eq5.11} hold, one gets
\[ y_{j+1} \leqslant \gamma 2^{\gamma j} \left(\frac{\nu_o k_1^{1-\lambda^+} r^{p^+}}{k_1^{1-\lambda_+} r^{p^+}}\right)^{\frac{p}{N+2/\alpha}} \epsilon^{-\frac{(\lambda_+ - \lambda_-)(N+p)}{N+2/\alpha}} \ y_j^{1+\frac{p}{N+2/\alpha}} \leqslant \gamma 2^{\gamma j} \nu_o^{\frac{p}{N+2/\alpha}} \epsilon^{-\frac{(\lambda_+ - \lambda_-)(N+p)}{N+2/\alpha}} \ y_j^{1+\frac{p}{N+2/\alpha}} \]
and then, for 
\[ \nu_0= \gamma^{-\frac{N+2/\alpha}{p}} \ 2^{-\gamma (\frac{N+2/\alpha}{p})^2} \ \epsilon^{\frac{(\lambda_+ - \lambda_-)(N+p)}{p}} \]
depending only on the data, one has 
\[ \left|K^{k/C_*}_{r}(y) \times (t_1, t_1 + \nu_o k_1^{1-\lambda_+} r^{p^+}] \cap [u<k_1] \right| \leqslant \left|K^{k/C_*}_{r}(y) \times (t_1, t_1 + \nu_o k_1^{1-\lambda_+} r^{p^+}] \right| \]
and therefore one can conclude that
\begin{equation}
    u(x,t)\geqslant \frac{k_1}{2}= \frac{\epsilon k}{2^{j_*+2} e^{\tau_1-\tau_0}},\quad (x,t) \in K^{k/C_*}_{r}(y) \times (t_1, t_1 + \nu_o k_1^{1-\lambda_+} r^{p^+}]\,.
    \end{equation}
We are one step away to prove Theorem \ref{th5.1}. So we start by choosing $\tau_0$ such that 
\[t_1 + \nu_o k_1^{1-\lambda_+} r^{p^+}= s + \delta k^{1-\lambda_+}r^{p^+}, \] that is, 
\[ e^{\tau_0}=\frac{2^{j_*+1}}{\epsilon}\Big(\frac{\delta}{\nu_0}\Big)^{\frac{1}{1-\lambda_{+}}}\]
and then we choose $C_*$ and $\epsilon_*$ as
\[ C_*= \frac{2^{j_*+2}}{\epsilon}\ e^{\tau_0+2 k^{1-\lambda_{+}}_0},\quad \epsilon_*=\dfrac{p_{-}}{N \log C_*}, \qquad \mbox{being} \ \ k_0=\epsilon e^{\tau_0} \ ; \] 
these choices assure that  \eqref{eq5.11} and \eqref{epsilonstarsing} hold true. 

\vspace{.2cm}

\noindent Theorem \ref{th5.1} is proved once we observe that 
$$t_1=s+\delta k^{1-\lambda_{+}} r^{p_{+}} (1-e^{-\tau_1(1-\lambda_{+})})\leqslant s+\delta k^{1-\lambda_{+}} r^{p_{+}} (1-e^{-(\tau_0+2 k^{1-\lambda_{+}}_0)(1-\lambda_{+})}),$$
and take $\varepsilon=e^{-(\tau_0+2 k^{1-\lambda_{+}}_0)(1-\lambda_{+})} \in (0,1)$, depending only on the data.

\section{Regularity results towards local H\"older continuity }\label{Sec.6}

\noindent Based on the results on the expansion of positivity, we obtain the reduction of the oscillation of the nonnegative, locally bounded, local weak solutions to \eqref{eq1.1}-\eqref{eq1.2} for specific ranges of $\lambda_i$ and $p_i$. The singular or degenerate character of \eqref{eq1.1} leads to different proofs either we are considering the fast diffusion range, $\lambda_+<1$, or the slow diffusion range, $\lambda_+>1$ (respectively). In the study of both cases we consider $(x_o,t_o)$ to be an interior point of $\Omega_T$ and define

\begin{equation}\label{defosc}
 \mu^+:= \esssup_{\Omega_T} u ; \qquad \mu^-:= \essinf_{\Omega_T}  u ; \qquad  \omega:= \essosc_{\Omega_T} u = \mu^+-\mu^- \ .
\end{equation}
We will consider $\mu^-=0$ since it is the interesting case to study and, for the sake of simplicity and by translation, we will consider $(x_o,t_o)=(0,0)$.

\vspace{.2cm}

\noindent Along this section,  we will be considering, for a large positive number $C_2$ (to be fixed), the extra assumption
\begin{equation}\label{newrestriction}
    p_+-p_- \leq \min\left\{ \frac{p_-}{\ln C_2} , \frac{p \ln 4}{N \ln C_2}  \right\}\ , 
\end{equation}
and the cube ($r>0$)
\begin{equation}\label{newcube}
\tilde{K}_{4r}^{\frac{\omega}{2C_2}} = \left\{x: |x_i|< \frac{(4r)^{\frac{p_+}{p_i}}}{\left(\frac{\omega}{2}\right)^{\frac{\lambda_+-\lambda_i}{p_i}}} C_2^{\frac{p_+-p_i}{p_i}} \right\} \ .
\end{equation}

\subsection{Reducing the oscillation within the fast diffusion range}

\noindent Let $r>0$ such that $K_{8r}^{\omega/2} \times (0, \theta r^{p_+}] \subset \Omega_T$, for $\theta= \displaystyle{\left(\frac{\omega}{2}\right)^{1-\lambda_+}}$ and $\lambda_+ <1$.

\vspace{.5cm}

\noindent Let $s \in (0, \theta r^{p_+}]$ be a time level (to be determined) such that, either
\begin{equation}\label{Alt1}
\left|K_{r}^{\omega/2} \cap [u(\cdot, s) > \omega/2] \right|\geq \frac{1}{2}\left|K_{r}^{\omega/2} \right|
\end{equation}
or
\begin{equation}\label{Alt2}
\left|K_{r}^{\omega/2} \cap [u(\cdot, s) > \omega/2] \right| < \frac{1}{2}\left|K_{r}^{\omega/2} \right|
\end{equation}

\noindent If \eqref{Alt1} holds, then by applying the result on the expansion of positivity Theorem \ref{th4.1} one gets
\begin{equation}\label{redoscAlt1}
u \geq \frac {\omega}{2C_1} \ , \quad \mbox{in} \ K^{\omega/(2C_{1})}_{r} \times [s+(1-\epsilon_1) \delta_1\, \theta \, r^{p_{+}}, s+ \delta_1 \theta \,r^{p_{+}}]
\end{equation}
where $0<\epsilon_1, \delta_1<1$ and $C_1>1$ depend only on the data. Remember that this holds for
\[ \lambda_+ <1 \qquad \mbox{and} \qquad \lambda_+-\lambda_- \leq \frac{p_{-}}{N \ln C_1} \ . \]

\vspace{.2cm}

\noindent If \eqref{Alt1} fails to happen then \eqref{Alt2} holds true: this is the interesting (and challenging) case to study. From \eqref{Alt2} and under assumption \eqref{newrestriction}, we can deduce that
\begin{equation}\label{Alt2a}
\left|\tilde{K}_{4r}^{\frac{\omega}{2C_2}} \cap [\bar{u}(\cdot, s) > \omega/2] \right| \geq \alpha_o \left|\tilde{K}_{4r}^{\frac{\omega}{2C_2}} \right| , \qquad \alpha_o= \frac{1}{4^{p_+\frac{N}{p} +\frac{3}{2}}}
\end{equation}
where $ 0 \leqslant\bar{u}:= \omega-u \leqslant \omega$ satisfies 
\begin{equation}\label{eq1.1a}
\bar{u}_t-\sum\limits_{i=1}^N \left(\bar{a}_i(x, t,\omega -\bar{u}, D \bar{u}) \right)_{x_i}=0,\quad (x,t)\in \Omega_T,
\end{equation} 
where the functions $\bar{a}_i(x, t, \omega - \bar{u}, D\bar{u}):= - a_i(x, t, u, Du)$, for $i=1, \cdots,N$. Observe that, for $\bar{u}$ close to zero, \eqref{eq1.1a} behaves like the anisotropic $p$-Laplacian 
\begin{equation}\label{anisopLap}
    \bar{u}_t-\sum\limits_{i=1}^N \left(\omega^{(m_i-1)(p_i-1)} |\bar{u}_{x_i}|^{p_i-2} \bar{u}_{x_i}\right)_{x_i}  = 0  \ . 
\end{equation}
Our goal now is to obtain results on the expansion of positivity of $\bar{u}$ satisfying \eqref{Alt2a}, under condition \eqref{newrestriction}. For that we start by deriving the energy estimates written for the test function $\varphi=( \bar{u}-\frac{\omega}{2})_- \xi^{p_+}(x)$ over $Q=\tilde{K}_{4r}^{\frac{\omega}{2C_2}} \times [s, s + \delta_2 \theta r^{p_+}]$, for some $\delta_2>0$,
\[ \sup_{s \leq t \leq s + \delta_2 \theta r^{p_+}} \int_{\tilde{K}_{4r}^{\frac{\omega}{2C_2}}} \left( \bar{u}-\frac{\omega}{2}\right)_-^2 \xi^{p_+} \ dx + \sum_{i=1}^N \iint_{Q} \omega^{(m_i-1)(p_i-1)} \left|{\left( \bar{u}-\frac{\omega}{2}\right)_-}_{x_i}\right|^{p_i} \xi^{p_+} \ dx \ dt   \]
\begin{eqnarray*}
&\leq &\int_{\tilde{K}_{4r}^{\frac{\omega}{2C_2}}} \left( \bar{u}-\frac{\omega}{2}\right)_-^2 \xi^{p_+}(x, s) \ dx + \gamma \iint_{Q} \omega^{(m_i-1)(p_i-1)}\left( \bar{u}-\frac{\omega}{2}\right)_-^{p_i} |\xi_{x_i}|^{p_i} \ . 
\end{eqnarray*}
Let $\sigma$ and $ \epsilon_2$ be two fixed small positive numbers (to be chosen). By considering, on the one hand, the time independent cutoff function $\xi(x) \in [0,1]$ defined in $\tilde{K}_{4r}^{\frac{\omega}{2C_2}}$ satisfying
\[ |\xi_{x_i}| \leq \gamma \frac{(\omega/2)^{\frac{\lambda_+-\lambda_i}{p_i}}}{\sigma r^{\frac{p_{+}}{p_i}} C_2^{\frac{p_+-p_-}{p_i} }} \]
and recalling \eqref{Alt2a}; and, on the other hand, the integration on the left hand side over the smaller cube $\tilde{K}_{4r\sigma}^{\frac{\omega}{2C_2}}\cap[\bar{u}< \epsilon_2 \omega/2]$, one gets 
\[ \left|\tilde{K}_{4r}^{\frac{\omega}{2C_2}} \cap [\bar{u}< \epsilon_2 \frac{\omega}{2}] \right| \leq \frac{1}{(1-\epsilon_2)^2} \left\{ 1-\alpha_o + \gamma \frac{\delta_2}{\sigma^{p_+}} + N \sigma \right\} |\tilde{K}_{4r}^{\frac{\omega}{2C_2}} | , \quad \forall t \in [s, s+ \delta_2 \theta r^{p_+}] \]
and therefore
\begin{equation*}
    |\tilde{K}_{4r}^{\frac{\omega}{2C_2}} \cap [\bar{u}< \epsilon_2 \omega/2]| \leq \left(1- \frac{\alpha_o}{2}\right) |\tilde{K}_{4r}^{\frac{\omega}{2C_2}}| , \quad \forall t \in [s, s+ \delta_2 \theta r^{p_+}]
\end{equation*}
once we choose
\[ \sigma = \frac{\alpha_o}{8N} \ , \quad \delta_2 = \frac{\alpha_o^{1+p_+}}{\gamma 2^{3+3p_+}N^{p_+} }  \ , \quad \epsilon_2 \leq 1- \sqrt{\frac{1-3\alpha_o/4}{1-\alpha_o/2}} \ .\]
We have just proven the first result towards expansion of positivity: assuming that \eqref{Alt2} holds true, and under restriction \eqref{newrestriction}, there exist $\epsilon_2, \delta_2 \in (0,1)$, depending only on the data, such that 
\begin{equation*}
    |\tilde{K}_{4r}^{\frac{\omega}{2C_2}}\cap [\bar{u} > \epsilon_2 \omega/2]| \geq \frac{\alpha_o}{2}  |\tilde{K}_{4r}^{\frac{\omega}{2C_2}} | , \qquad \forall t \in [s, s+ \delta_2 \theta r^{p_+}] \ .
\end{equation*}
If necessary, we consider $\delta_1$,  obtained along the process of the expansion of positivity related to the study of the first alternative \eqref{Alt1}, small enough so that $\delta_1 \leq \delta_2$. Then we also have 
\begin{equation}\label{1stexp}
    |\tilde{K}_{4r}^{\frac{\omega}{2C_2}}\cap [\bar{u} > \epsilon_2 \omega/2]| \geq \frac{\alpha_o}{2}  |\tilde{K}_{4r}^{\frac{\omega}{2C_2}} | , \qquad \forall t \in [s, s+ \delta_1 \theta r^{p_+}] \ .
\end{equation}

In order to expand information \eqref{1stexp} to a full cylinder, we consider to be working within the singular range $p_+<2$, and introduce the new function and variables
\[v(z,\tau)= \frac{e^{\tau}}{\omega/2} \bar{u} (x,t) , \quad z_i= \frac{x_i}{r^{\frac{p_+}{p_i}} } \left(\frac{\omega}{2}\right)^{\frac{\lambda_+-\lambda_i-}{p_i}}  \quad t=s + \delta_1 \theta r^{p_+} \left(1-e^{-\tau(2-p_+)}\right) \]
 which due to \eqref{eq1.1a} satisfies
\begin{equation}\label{eqv} 
v_{\tau} \geqslant (2-p_+)\ \delta_1 \ \sum\limits_{i=1}^N \left( e^{\tau(p_+-p_i)} 2^{(m_i-1)(p_i-1)} |v_{z_i}|^{p_i-2} v_{z_i}\right)_{z_i}   
\end{equation}
and transforms the measure theoretical information \eqref{1stexp} into a new one
\begin{equation}\label{1stexpa}
    |\tilde{K}_{4}^{\frac{1}{C_2}}\cap [v(\cdot, \tau) > \epsilon_2 \e^{\tau}]| \geq \frac{\alpha_o}{2}  |\tilde{K}_{4}^{\frac{1}{C_2}} | , \qquad \forall \tau \geqslant 0 \ .
\end{equation}
In a quite similar way to what was done for $\lambda_+<1$, we fix $\tau_0 >0 $  and consider the levels
\[ k_0=\epsilon_2 e^{\tau_0}\quad \text{and}\quad k_j=\frac{k_0}{2^j},\quad j=0, 1, 2,...,s_2-1, \]
where $s_2$ is to be specified later, to get
\begin{equation}\label{1stexpa}
    |\tilde{K}_{4}^{\frac{1}{C_2}}\cap [v(\cdot, \tau) > k_j]| \geq \frac{\alpha_o}{2}  |\tilde{K}_{4}^{\frac{1}{C_2}} | , \qquad \forall \tau \geqslant \tau_o \ .
\end{equation}
Consider
\begin{equation} \label{condC2}
2^{s_2} \leq e^{\tau_o} \qquad \mbox{and} \qquad e^{\tau_o + k_0^{2-p_+}} \leq C_2 \ .
\end{equation}
By testing \eqref{eqv} with the functions $(v-k_j)_- \xi^{p_+}$ over \[ Q=\tilde{K}_{4}^{\frac{1}{C_2}} \times [\tau_o , \tau_o + 2e^{\tau_o(2-p_+)}] \supset  \tilde{K}_{2}^{\frac{1}{C_2}} \times [\tau_o + e^{\tau_o(2-p_+)}, \tau_o + 2e^{\tau_o(2-p_+)}]=Q_{1/2} \] we arrive at
\[\sum_{i=1}^N \iint_{Q_{1/2}} |(v-k_j)_-{z_i}|^{p_i}  \ dz \ \d\tau \leq \gamma \ k_j^{p_i} \ |Q_{1/2}|
\]
and then, recalling \eqref{newrestriction} and applying the isoperimetric result, we finally get
\[  \left|Q_{1/2} \cap \left[v< \epsilon_2 \ \frac{e^{\tau_o}}{2^{s_2}}\right]\right| \leq \frac{\gamma}{s_2^{\frac{p_- -1}{p_-}}} |Q_{1/2}| \ . \]
This proves that, assuming \eqref{1stexp} and under assumptions $p_+<2$, \eqref{newrestriction} and \eqref{condC2}, for every small $\nu \in (0,1)$, there exists $s_2$ depending only on the data, such that 
\begin{equation}\label{ineqnu}  \left|Q_{1/2} \cap \left[v< \epsilon_2 \ \frac{e^{\tau_o}}{2^{s_2}}\right]\right| \leq \nu \  |Q_{1/2}| \ .
\end{equation}
Our next step is to fix $\nu \in (0,1)$, determining therefore $s_2$. 

\noindent For the time being, assume that $\nu$ is known, then from \eqref{ineqnu} one has
\[  \left|Q_{n} \cap \left[v< \epsilon_2 \ \frac{e^{\tau_o}}{2^{s_2}}\right]\right| \leq \nu \  |Q_{n}| \]
at least in one of the subcylinders
\[Q_n=\tilde{K}_{2}^{\frac{1}{C_2}} \times \left[T_o^n, T_o^n+ \left(\frac{k_0}{2^{s_2}}\right)^{2-p_+} \right] , \quad T_o^n= \tau_o + k_0^{2-p_+} + n \left(\frac{k_0}{2^{s_2}}\right)^{2-p_+}\]
for some $n=0,\dots, 2^{s_2(2-p_+)}-1$. By testing \eqref{eqv} with $(v-k_j)_- \xi^p_+$ over $Q_{n,j}$ where for $j=0,1,\dots$
\[k_j= \frac{k_0}{2^{s_2+1}} \left(1+\frac{1}{2^j}\right), \qquad r_j= 1+\frac{1}{2^j}\]
\[Q_{n,j}= \tilde{K}_{r_j}^{\frac{1}{C_2}} \times \left[T_o^n + \frac{1}{2} \left(1-\frac{1}{2^j}\right) \left(\frac{k_0}{2^{s_2}}\right)^{2-p_+}, T_o^n+ \left(\frac{k_0}{2^{s_2}}\right)^{2-p_+} \right] = \tilde{K}_{r_j}^{\frac{1}{C_2}} \times I_{n,j}\]
and $\xi$ is a piecewise smooth cutoff function entrapped between $Q_{n,j+1} \subset Q_{n,j}$, we arrive at
\[ \sup_{ \tau \in I_{n,j}}  \int_{\tilde{K}_{r_j}^{\frac{1}{C_2}}} \left( v-k_j\right)_-^2 \xi^{p_+} \ dz + \sum_{i=1}^N \iint_{Q_{n,j}}  \left|((v-k_j)_-)_{z_i}\right|^{p_i} \xi^{p_+} \ dz \ d\tau \]
\begin{eqnarray}
    &\leqslant & \gamma \ 2^{\gamma j}  \ \left(\frac{k_0}{2^{s_2}}\right)^{p_+} \left[1 + \sum_{i=1}^N \left( \frac{e^{\tau_o + k_0^{2-p_+}} C_2 }{\frac{k_0}{2^{s_2}}} \right)^{p_+-p_i} \right] |A_j| \nonumber \\
    & = & \gamma 2^{\gamma j}  \ \left(\frac{k_0}{2^{s_2}}\right)^{p_+} \left[1 + \sum_{i=1}^N \left( \frac{ 2^{s_2} e^{ k_0^{2-p_+}} C_2 }{\epsilon_2} \right)^{p_+-p_i} \right] |A_j|  \leqslant \gamma \ 2^{\gamma j}  \left(\frac{k_0}{2^{s_2}}\right)^{p} |A_j| \label{new}
\end{eqnarray}
once we consider $|A_j|:= |Q_{n,j} \cap [v<k_j]|$ and recall assumptions \eqref{newrestriction} and \eqref{condC2}. Proceeding in a now standard way that makes use of H\"older's inequality and the anisotropic embedding one gets
\[ |A_{j+1}| \leq \gamma  \ 2^{\gamma j}  \left(\frac{k_0}{2^{s_2}}\right)^{\frac{p-2}{N+2}} |A_j|^{1+ \frac{1}{N+2}} \quad \Longrightarrow \quad  Y_{j+1} \leq \gamma  \ 2^{\gamma j} \ Y_j^{1+ \frac{1}{N+2}}\ \quad \mbox{for} \ \  Y_j=\dfrac{|A_j|}{|Q_{n,j}|}  \ .\]
By choosing $\nu= \gamma^{-(N+2)} \ 2^{-\gamma(N+2)^2}$, we can apply the fast geometric convergence lemma to conclude that
\[ v (z,\tau)\geqslant \frac{k_0}{2^{s_2} +1} \qquad  (z,\tau) \in \tilde{K}_{2}^{\frac{1}{C_2}} \times \left[T_o^n + \frac{1}{2}\left(\frac{k_0}{2^{s_2}}\right)^{2-p_+} , T_o^n+ \left(\frac{k_0}{2^{s_2}}\right)^{2-p_+} \right] \]
in particular, for $\tau_o + k_0^{2-p_+} < T_o^n + \frac{1}{2}\left(\frac{k_0}{2^{s_2}}\right)^{2-p_+}  < \tau_1 \leqslant T_o^n+ \left(\frac{k_0}{2^{s_2}}\right)^{2-p_+} < \tau_o + 2k_0^{2-p_+}$,
\[ v (z,\tau_1)\geqslant \frac{k_0}{2^{s_2} +1} = \frac{ \epsilon_2 \ e^{\tau_o}}{2^{s_2}+1} \ , \qquad  z \in \tilde{K}_{2}^{\frac{1}{C_2}} \]
and, from here, we also get
\[ \bar{u} (x,t_1)\geqslant \frac{k_0}{2^{s_2} +1} = \epsilon_2 \frac{e^{\tau_o-\tau_1}}{2^{s_2}+1} \frac{\omega}{2} \geqslant \frac{\epsilon_2 \ e^{-\tau_1}}{4}\ \omega  =: a \omega \ , \qquad  x \in K_{r}^{\frac{\omega}{2}} \]
for $ t_1= s + \delta_1 \theta (1-e^{-(2-p_+)\tau_1}) r^{p_+} $. 

\vspace{.2cm}

\noindent In order to get to the final step - the reduction of the oscillation - we make a similar reasoning, with respect to $\bar{u}$, as in Lemma \ref{lem2.5}, a variant of De Giorgi-type Lemma with initial data; for that purpose, we need to use the full anisotropy of \eqref{anisopLap}. 

\vspace{.2cm}

\noindent Let $m= \min\{1, m_1, \cdots, m_N\}$ and test \eqref{anisopLap} with the functions $(\bar{u}^{m}-k_j^m)_- \xi^{p_+}$ over $Q_j$, where
\[k_j= \frac{a \omega}{2} \left(1+\frac{1}{2^j}\right) \ , \qquad r_j= \frac{r}{2} \left(1+\frac{1}{2^j}\right) \ , \qquad Q_j= K_{r_j}^{\frac{\omega}{2}} \times [t_1, t_1 + b r^{p_+} ]\]
and $0\leqslant \xi(x)\leqslant 1$ is a time-independent cutoff function defined in $K_j$ and such that $|\xi_{x_i}| \leq \dfrac{\gamma 2^{\gamma j}}{r^{p_+}}\left(\dfrac{\omega}{2}\right)^{\lambda_+-\lambda_i} $, deducing
\[ \sup_{ t_1 \leq t \leq t_1 + br^{p_+}}  \int_{K_{r_j}^{\frac{\omega}{2}}} g_{-}( \bar{u}^m, k_j^m) \xi^{p_+} \ dz + \sum_{i=1}^N \iint_{Q_j}  \omega^{(m_i-1)(p_i-1)} \ \bar{u}^{(1-m)(p_i-1)}\left|((\bar{u}^{m}-k_j^m)_-)_{x_i}\right|^{p_i} \xi^{p_+} \ dx \ dt \]
\begin{eqnarray*}
    &\leqslant & \gamma \sum_{i=1}^N \iint_{Q_j}  \omega^{(m_i-1)(p_i-1)} \bar{u}^{(1-m)(p_i-1)}(\bar{u}^m-k_j^m)_{-}^{p_i} \ |\xi_{x_i}|^{p_i} \ dx \ dt  \\
    &\leqslant & \gamma \ \frac{2^{\gamma j}}{r^{p_+}}  \ (a \omega)^{\lambda_+ + m}  \
   \sum_{i=1}^N a^{p_i-1-\lambda_+} |Q_j \cap [\bar{u}^{m}<k_j^m)]|\\
   &\leqslant & \gamma \ \frac{2^{\gamma j}}{r^{p_+}}  \ (a \omega)^{\lambda_+ + m}  \
   \sum_{i=1}^N a^{p_i-1-\lambda_+} |Q_j \cap [\tilde{u}^{m}<k_j^m)]| \ , \qquad \tilde{u}= \max \left\{ \bar{u}, \dfrac{a\omega}{2} \right\} \ .
\end{eqnarray*}
As for the left hand side, 
\[ \int_{K_{r_j}^{\frac{\omega}{2}}} g_{-}( \bar{u}^m, k_j^m) \xi^{p_+} \ dz \geq \int_{K_{r_j}^{\frac{\omega}{2}}} g_{-}( \tilde{u}^m, k_j^m) \xi^{p_+} \ dz \geq \frac{(a\omega)^{1-m}}{2^{2-m}} \int_{K_{r_j}^{\frac{\omega}{2}}} ( \tilde{u}^m - k_j^m)_-^2 \xi^{p_+} \ dz\] 
\[  \iint_{Q_j}  \omega^{(m_i-1)(p_i-1)} \bar{u}^{(1-m)(p_i-1)}(\bar{u}^m-k_j^m)_{-}^{p_i} \ |\xi_{x_i}|^{p_i} \ dx \ dt  \geq   a^{(1-m_i)(p_i-1)}\iint_{Q_j} \left|((\tilde{u}^{m}-k_j^m)^{\alpha_i}_-)_{x_i}\right|^{p_i} \xi^{p_+} \ dx \ dt  \]
for $\alpha_i= \dfrac{m_i(p_i-1) + m}{p_i m}>1 $;  collecting all the previous estimates we arrive at
\[ \sup_{ t_1 \leq t \leq t_1 + br^{p_+}}  (a\omega)^{1-m} \int_{K_{r_j}^{\frac{\omega}{2}}} ( \tilde{u}^m - k_j^m)_-^2 \xi^{p_+} \ dx + \gamma \sum_{i=1}^N a^{(1-m_i)(p_i-1)}\iint_{Q_j} \left|((\tilde{u}^{m}-k_j^m)^{\alpha_i}_-)_{x_i}\right|^{p_i} \xi^{p_+} \ dx \ dt \]
\[\leq \gamma \ \frac{2^{\gamma j}}{r^{p_+}}  \ (a \omega)^{\lambda_+ + m}  \
   \sum_{i=1}^N a^{p_i-1-\lambda_+} |A_j| \ , \qquad \mbox{for} \quad  |A_j|=|Q_j \cap [\tilde{u}^{m}<k_j^m)]| \ . \]
By applying H\"older's inequality together with the anisotropic embedding \eqref{PS}, we obtain
\begin{eqnarray*}
    (k_j^m-k_{j+1}^m)^{\alpha p} |A_{j+1}| & \leqslant & \iint_{Q_j}(( \tilde{u}^m -k_j^m)_- \xi^{p_+})^{\alpha p} \ dx \ dt  \leqslant  \gamma \left( \int_{K_{r_j}^{\frac{\omega}{2}}} ( \tilde{u}^m - k_j^m)_-^2 \xi^{p_+} \ dx\right)^{\frac{\alpha p }{\alpha N+2} }\\
     & & \times  \left(\sum_{i=1}^N \iint_{Q_j} \left|(((\tilde{u}^{m}-k_j^m)_- \xi^{p_+})^{\alpha_i}_-)_{x_i}\right|^{p_i}  \ dx \ dt\right)^{\frac{\alpha N}{\alpha N+2} }  |A_j|^{1-\frac{\alpha N}{\alpha N+2} } \\
     \ \  
     & \leqslant & \gamma 2^{\gamma j} \left(\frac{(a\omega)^{\lambda_++m}}{r^{p_+}}\right)^{\frac{\alpha(p+N)}{\alpha N+2}}\left((a \omega)^{m-1} \sum_{i=1}^N a^{p_i-1-\lambda_+}\right)^{\frac{\alpha p}{\alpha N+2}}
     \left(\sum_{i=1}^N a^{p_i-p_+-\lambda_+}\right)^{\frac{\alpha N}{\alpha N+2}} \\
     & & |A_j|^{1+\frac{\alpha p}{\alpha N+2} }
\end{eqnarray*}
thereby 
\[Y_{j+1} \leq \gamma 2^{\gamma j} \left[\left(\frac{\omega}{2}\right)^{1-\lambda_+} a^{\frac{2(N+p}{p}} \right]^{- \frac{\alpha p}{\alpha N+2}}  b^{\frac{\alpha p}{\alpha N+2}} \ Y_j^{1+\frac{\alpha p}{\alpha N+2}} \ , \qquad \mbox{for} \qquad Y_j= \frac{|A_j|}{|Q_j|} \ . \]
By choosing $b= \gamma^{-\frac{\alpha N+2}{\alpha p}} 2^{-\gamma (\frac{\alpha N + 2}{\alpha p})^2} \  a^{\frac{2(N+p)}{p}} \left(\frac{\omega}{2}\right)^{1-\lambda_+} = \nu_1 \  a^{\frac{2(N+p)}{p}} \left(\frac{\omega}{2}\right)^{1-\lambda_+}$ we can then apply the fast geometric convergence lemma and obtain the lower estimate for $\bar{u}$
\[ \bar{u} \geq \frac{a \omega}{2} , \quad K_{r/2}^{\frac{\omega}{2}} \times [t_1, t_1+ b r^{p_+}] \]
and therefore, being $Q= K_{r/2}^{\frac{\omega}{2}} \times [s + \delta_1 (1-e^{-(2-p_+)\tau_1}) \theta\  r^{p_+}, s + (\delta_1  (1-e^{-(2-p_+)\tau_1}) + \nu_1 a^{\frac{2(N+p)}{p}} ) \theta \ r^{p_+} ] $
\[ \essosc_{Q} u \leq \left(1- \frac{a}{2}\right) \omega \ . \]
We now choose $\tau_1$ big enough so that that $e^{-(2-p_+)\tau_1} \leq \epsilon_1$ and take 
\[ \delta= \min \left\{ \delta_1, \delta_1  (1-e^{-(2-p_+)\tau_1}) + \nu_1 a^{\frac{2(N+p)}{p}}  \right\} \]
and finally take
\[s= (1-\delta) \theta r^{p_+} \ .\]
From this choices and the study of both alternatives \eqref{Alt1}-\eqref{Alt2}, we derive
\begin{equation}
    \essosc_{K_{\frac{r}{2}}^{\frac{\omega}{2}} \times  \left( \left(1-\delta + (1-e^{-(2-p_+)\tau_1})\delta_1\right) \theta r^{p_+}, \theta r^{p_+} \right]} u \leq \eta \ \omega \ , \qquad \eta= \min \left\{ 1-\frac{a}{2}, 1-\frac{1}{2C_1} \right\}
\end{equation}

\subsection{Reducing the oscillation within the slow diffusion range}

\noindent Let $r>0$ such that $K_{8r}^{\omega/2} \times (0, \theta r^{p_+}] \subset \Omega_T$, for $\theta= A \displaystyle{\left(\frac{\omega}{2}\right)^{1-\lambda_+}}$, $A>1$ to be chosen and $\lambda_+ > 1$.

\vspace{.5cm}

\noindent Let $s \in (0, \theta r^{p_+}]$ be a time level (to be determined) such that, either
\begin{equation}\label{alt1}
\left|K_{r}^{\omega/2} \cap [u(\cdot, s) > \omega/2] \right|\geq \frac{1}{2}\left|K_{r}^{\omega/2} \right|
\end{equation}
or
\begin{equation}\label{alt2}
\left|K_{r}^{\omega/2} \cap [u(\cdot, s) > \omega/2] \right| < \frac{1}{2}\left|K_{r}^{\omega/2} \right|\,.
\end{equation}

\noindent If \eqref{alt1} holds, then by applying the result on the expansion of positivity Theorem \ref{th4.1} one gets
\begin{equation}\label{redoscalt1}
u \geq \frac {\omega}{2C_1} \ , \quad \mbox{in} \ K^{\omega/(2C_{1})}_{2r} \times (s+B_1 \left(\frac{\omega}{2}\right)^{1-\lambda_+} r^{p_+} ,  s+ \tilde{B_1}\left(\frac{\omega}{2}\right)^{1-\lambda_+} r^{p_+}]
\end{equation}
where  
\[ B_1= \delta_1  e^{3\tau_{1}(\lambda_{+}-1)/4} < \tilde{B_1}= \delta_1  e^{\tau_{1}(\lambda_{+}-1)} \ , \qquad C_1= \dfrac{2^{1+s_1} e^{\tau_{1}}}{\epsilon_1} \ , \quad \tau_1= \left( \frac{2^{s_1}}{\epsilon_1}\right)^{\lambda_+-1}\]
depend only on the data. Remember that this holds for
\[ \lambda_+ >1 \qquad \mbox{and} \qquad \lambda_+-\lambda_- \leq \frac{p_-}{N \ln C_1}\]

\vspace{.2cm}

\noindent Consider that \eqref{alt1} does not hold, hence \eqref{alt2} is in force. Then, by a very close reasoning to what was presented for the fast diffusion range, considering $ 0 \leqslant\bar{u}:= \omega-u \leqslant \omega$ which, close to zero, satisfies the anisotropic differential equation \eqref{anisopLap}, the extra assumption \eqref{newrestriction} and the "new" cube as presented in \eqref{newcube}, one has
\[\left|\tilde{K}_{4r}^{\frac{\omega}{2C_2}} \cap \left[\bar{u}(\cdot, s) > \frac{\omega}{2}\right] \right| \geq \alpha_o \left|\tilde{K}_{4r}^{\frac{\omega}{2C_2}} \right| , \qquad \alpha_o= \frac{1}{4^{p_+\frac{N}{p} +\frac{3}{2}}}
\]
and also 
\begin{equation}\label{newinfo}
\left|\tilde{K}_{4r}^{\frac{\omega}{2C_2}} \cap \left[\bar{u}(\cdot, s) > e^{-\tau} \ \frac{\omega}{2}\right] \right| \geq \alpha_o \left|\tilde{K}_{4r}^{\frac{\omega}{2C_2}} \right| , \qquad \forall \tau \geq 0  \ . 
\end{equation}

\noindent We can then extend this information in time. Namely,

\begin{lemma}\label{lem4.1a}
Assume \eqref{newinfo} holds. There exist $\delta_2$, $\epsilon_2 \in (0,1)$, depending only on the data and on $\alpha_0$, such that for any $0< \tau \leqslant \tau_{2}\leqslant \ln C_{2}$ 
 \begin{equation}\label{eq2alta}
\left|\tilde{K}_{4r}^{\frac{\omega}{2C_2}} \cap \left[\bar{u}(\cdot, s) > \epsilon_2 \ e^{-\tau} \ \frac{\omega}{2}\right] \right| \geq \frac{\alpha_o}{2} \left|\tilde{K}_{4r}^{\frac{\omega}{2C_2}} \right|  \qquad s <t \leqslant s+ \delta_2 \left(\frac{\omega}{2}\right)^{1-\lambda_+} e^{\tau(p_+-2)} r^{p_+} \ .
\end{equation}
\end{lemma}

\begin{proof}
By testing \eqref{anisopLap} with $\left(\bar{u} -e^{-\tau} \ \dfrac{\omega}{2} \right)_- \xi(x)^p_+$ over $ \tilde{K}_{4r}^{\frac{\omega}{2C_2}} \times (s, s+ \delta_2 \left(\frac{\omega}{2}\right)^{1-\lambda_+} e^{\tau(p_+-2)} r^{p_+}]$, where $\xi$ is a time independent cutoff function trapped between $\tilde{K}_{4r(1-\sigma)}^{\frac{\omega}{2C_2}} \subset \tilde{K}_{4r}^{\frac{\omega}{2C_2}}$, for $\sigma \in (0,1)$ to be chosen, and verifying $ |\xi_{x_i}|^{p_i} \leq \frac{\left(\frac{\omega}{2}\right)^{\lambda_+-\lambda_i}}{\sigma^{p_+} r^{p_+} C_2^{p_+-p_i}}$, we get 

 $\forall \ \  s <t \leqslant s+ \delta_2 \left(\frac{\omega}{2}\right)^{1-\lambda_+} e^{\tau(p_+-2)} r^{p_+}$,
\begin{eqnarray*} 
\left(e^{-\tau} \ \frac{\omega}{2}\right)^2 (1-\epsilon_2)^2 \left| \tilde{K}_{4r(1-\sigma)}^{\frac{\omega}{2C_2}} \cap \left[\bar{u} < \epsilon_2 \ e^{-\tau} \ \dfrac{\omega}{2} \right] \right| & \leqslant & \int_{\tilde{K}_{4r}^{\frac{\omega}{2C_2}}} \left( \bar{u}- e^{-\tau} \ \frac{\omega}{2}\right)_-^2 \xi^{p_+} \ dx   \\
& \leqslant &\left(e^{-\tau} \ \frac{\omega}{2}\right)^2 \left[ 1-\alpha_o + \gamma \frac{\delta_2}{\sigma^{p_+}} \sum_{i=1}^N \left(\frac{e^\tau}{C_2}\right)^{p_+-p_i}\right] \left|\tilde{K}_{4r}^{\frac{\omega}{2C_2}}\right| \\
& \leqslant &
\left(e^{-\tau} \ \frac{\omega}{2}\right)^2 \left[ 1-\alpha_o + \gamma \frac{\delta_2}{\sigma^{p_+}} \right] \left|\tilde{K}_{4r}^{\frac{\omega}{2C_2}}\right| \ . 
\end{eqnarray*}
Hence, for all $s <t \leqslant s+ \delta_2 \left(\frac{\omega}{2}\right)^{1-\lambda_+} e^{\tau(p_+-2)} r^{p_+}$,
\[ \left| \tilde{K}_{4r}^{\frac{\omega}{2C_2}} \cap \left[\bar{u} < \epsilon_2 \ e^{-\tau} \ \dfrac{\omega}{2} \right] \right| \leqslant \frac{1}{(1-\epsilon_2)^2 }
 \left( 1-\alpha_o + \gamma \frac{\delta_2}{\sigma^{p_+}} + N\sigma \right)  \left|\tilde{K}_{4r}^{\frac{\omega}{2C_2}}\right|  \ .
\] 
The proof is complete once we choose 
\[ \sigma = \frac{\alpha_o}{8N} \ , \quad \delta_2 = \frac{\alpha_o^{1+p_+}}{\gamma 2^{3+3p_+}N^{p_+} }  \ , \quad \epsilon_2 \leq 1- \sqrt{\frac{1-3\alpha_o/4}{1-\alpha_o/2}} \ .\]
\end{proof}

\noindent The measure theoretical information \eqref{eq2alta} now reads
\begin{equation}\label{eq2altav}
\left|\tilde{K}_{4}^{\frac{1}{C_2}} \cap \left[v(\cdot, \tau) > \epsilon_2\right] \right| \geq \frac{\alpha_o}{2} \left|\tilde{K}_{4}^{\frac{1}{C_2}} \right|  \qquad \forall \tau >0 \ , 
\end{equation}
where
\[v(z,\tau)= \frac{e^{\tau}}{\frac{\omega}{2}} \ \bar{u} (x,t) , \quad z_i= \frac{x_i}{r^{\frac{p_+}{p_i}} } \left(\frac{\omega}{2}\right)^{\frac{\lambda_+-\lambda_i-}{p_i}}  \quad t= s + \delta_2 \left(\frac{\omega}{2}\right)^{1-\lambda_+} \ e^{\tau(p_+-2)} \ r^{p_+} \]
 and $v$ satisfies
\begin{equation}\label{eqvnew} 
v_{\tau} \geqslant (p_+ -2)\ \delta_2 \ \sum\limits_{i=1}^N \left( e^{\tau(p_+-p_i)} 2^{(m_i-1)(p_i-1)} |v_{z_i}|^{p_i-2} v_{z_i}\right)_{z_i}    \ . 
\end{equation}
Under the assumptions that lead us to \eqref{eq2altav}, we have
\begin{lemma}\label{lem4.2n}
For any $\nu\in(0,1)$, there exists a positive number $s_2$, depending on the data, $\alpha_0$ and $\nu$, such that
\begin{equation}\label{dets2}
|\tilde{Q}\cap [v\leqslant \frac{\epsilon}{2^{s_2}}]|\leqslant \nu |\tilde{Q}|,\quad \tilde{Q}:=\tilde{K}^{1/C_{2}}_{4}\times\left[\frac{\tau_{2}}{2},\tau_{2}\right],\quad \tau_{2}:=\left(\frac{2^{s_2}}{\epsilon_2}\right)^{p_+-2},
\end{equation}
provided that
\begin{equation}\label{extra}
p_+ >2 \quad \mbox{and} \quad p_+-p_- \leqslant \frac{p_-}{\ln C_2} \ . 
\end{equation}
\end{lemma}

\begin{proof}
The proof is based on the isoperimetric result together with energy estimates derived for $v$.  Let $\dfrac{\tau_2}{2}\leqslant \tau \leqslant \tau_2$, $k_j= \dfrac{\epsilon_2}{2^j}$, for $j=0, \cdots, s_2-1$, and define
\[ A_j (\tau)= \tilde{K}_{4}^{1/C_2}\cap [v(\cdot,\tau)< k_j] \qquad \mbox{and} \qquad |A_j|= \int_{\frac{\tau_2}{2}}^{\tau_2} |A_{j}(\tau)|\ d\tau \ . \]
From the isoperimetric result and \eqref{eq2altav} one gets 
\[ \frac{k_j}{2} |A_{j+1}(\tau)| \leqslant \gamma(\alpha_0) C_2^{\frac{\lambda+-\lambda_-}{p_-}} \sum\limits_{i=1}^N \int\limits_{\tilde{K}_{4}^{1/C_2}\cap [k_{j+1}< v(\cdot, \tau)< k_j]}| v_{z_i}| \; dz \ , \]
after which one integrates in time over $\left[\dfrac{\tau_2}{2}, \tau_2\right] $ and applies H\"older's inequality 
\begin{eqnarray*}
\frac{k_j}{2}  |A_{j+1}|  & \leqslant & \gamma(\alpha_0) C_2^{\frac{\lambda+-\lambda_-}{p_-}} \left( \sum\limits_{i=1}^N  \iint_{\tilde{Q}} | ((v-k_j)_-)_{z_i}|^{p_i} \; dz\right)^{\frac{1}{p_i}} \ \left |A_j\setminus A_{j+1}\right|^{1-\frac{1}{p_i}}  \\
& \leqslant & \gamma(\alpha_0) \left( \frac{\gamma}{\delta_2} \ k_j^{p_i} \ |\tilde{Q}|\right)^{\frac{1}{p_i}}  \ \left |A_j\setminus A_{j+1}\right|^{1-\frac{1}{p_i}} \ . 
\end{eqnarray*}
The last inequality results from  energy estimates for $v$: in fact, when testing \eqref{eqvnew} with $(v-k_j)_- \xi^{p_+}$ over $Q:=\tilde{K}^{1/C_{2}}_{8}\times\left(0,\tau_{2}\right]$, for $0\leqslant \xi \leqslant 1 $ satisfying $x_i =1$ in $\tilde{Q}$, $|\xi_\tau| \leqslant \dfrac{\gamma}{\tau_2}$ and $|\xi_{z_i}|^p_i \leqslant \dfrac{\gamma }{C_2^{p_+-p_i}}$ one arrives at
\begin{eqnarray*}
 \sum\limits_{i=1}^N  \iint_{\tilde{Q}} | ((v-k_j)_-)_{z_i}|^{p_i} \; dz & \leqslant & \sum\limits_{i=1}^N  \iint_{Q} e^{\tau(p_+-p_i)} | ((v-k_j)_-)_{z_i}|^{p_i} \xi^{p_+} \; dz  \\
 & \leqslant & \frac{\gamma}{\delta_2}\iint_{Q} (v-k_j)_-^2 |\xi_{\tau}| \ dz \ d\tau + \gamma \iint_{Q} e^{\tau(p_+-p_i)}(v-k_j)_-^{p_i} |\xi_{z_i}|^{p_i} \ dz \ d\tau  \\
  & \leqslant & \frac{\gamma}{\delta_2} \  k_j^{p_i} \ |\tilde{Q}|
\end{eqnarray*}
recalling that $\tau_{2}:=\left(\dfrac{2^{s_2}}{\epsilon_2}\right)^{p_+-2}$, $\tau_2 \leq \ln C_2$ and assuming \eqref{extra}. 

\noindent Observe that from the previous inequality obtained for $|A_{j+1}|$ we get
\[ \frac{|A_{j+1}|}{|\tilde{Q}|} \leqslant \gamma(\alpha_o, \delta_2) \left( \frac{\left |A_j\setminus A_{j+1}\right|}{|\tilde{Q}|}\right)^{1-\frac{1}{p_i}} \leqslant \gamma(\alpha_o, \delta_2) \left( \frac{\left |A_j\setminus A_{j+1}\right|}{|\tilde{Q}|}\right)^{1-\frac{1}{p_-}}  , \quad j=0, \cdots, s_2-1\]
and so, by taking the power $\dfrac{p_-}{p_- -1}$ and summing up, we finally derive
\[ s_2 \left(\frac{|A_{s_2}|}{|\tilde{Q}|}\right)^{\frac{p_-}{p_- -1} } \leqslant 
\gamma(\alpha_o, \delta_2)  \sum_{j=0}^{s_2-1} \frac{\left |A_j\setminus A_{j+1}\right|}{|\tilde{Q}|} \leqslant \gamma(\alpha_o, \delta_2) \ . 
\]
We then choose $s_2$ such that $\frac{\gamma}{s_2^{\frac{p_--1}{p_-}}} \leq \nu$, which completes the proof.
\end{proof}

\noindent Our final aim is to fix $\nu \in (0,1)$, depending only on the data, and thereby determined $s_2$ (and consequently $\tau_2$). For that purpose we consider: sequences of cylinders $Q_j$,  $\tilde{K_2}^{\frac{1}{C_2}} \times (\frac{3}{4} \tau_2, \tau_2] \subset Q_j \subset \tilde{K_2}^{\frac{1}{C_2}} \times (\frac{1}{2} \tau_2, \tau_2]$ and of levels $k_j= \frac{\epsilon_2}{2^{s_2+1}} \left(1+\frac{1}{2^j}\right)$; cutoff functions $\xi$ vanishing on the parabolic boundary of $Q_j$ and verifying $|\xi_\tau| \leqslant \frac{\gamma}{\tau_2}$ and $|\xi_{z_i}|^{p_i} \leqslant \frac{\gamma}{C_2^{p_+-p_i}} $, and by testing \eqref{eqv} with $(v-k_j)_- \xi^p_+$ over $Q_j$, for $j=0,1,\dots$, we obtain
\[ \left( \frac{\epsilon_2}{2^{s_2}}\right)^{2-p_+} \sup_{ \frac{1}{2} \tau_2 \leqslant \tau \leqslant \tau_2  }\int_{\tilde{K}_{j}^{\frac{1}{C_2}}} \left(\left( v-k_j\right)_-\xi\right)^{p_+} \ dz + \sum_{i=1}^N \iint_{Q_{j}}  \left|((v-k_j)_-)_{z_i}\right|^{p_i} \xi^{p_+} \ dz \ d\tau \]
\begin{eqnarray}
& \leqslant & \sup_{ \frac{1}{2} \tau_2 \leqslant \tau \leqslant \tau_2} \int_{\tilde{K}_{J}^{\frac{1}{C_2}}} \left( v-k_j\right)_-^2 \xi^{p_+} \ dz + \sum_{i=1}^N \iint_{Q_{j}} e^{\tau(p_+-p_i)} \left|((v-k_j)_-)_{z_i}\right|^{p_i} \xi^{p_+} \ dz \ d\tau \nonumber \\
    &\leqslant & \gamma \ 2^{\gamma j}  \ \left(\frac{\epsilon_2}{2^{s_2}}\right)^{p_+} \left[1 + \sum_{i=1}^N \left( \frac{e^{\tau_2} \ 2^{s_2+1} }{\epsilon_2 \ C_2} \right)^{p_+-p_i} \right] |A_j| \ , \qquad |A_j|= |Q_j \cap [v<k_j]|\nonumber \\
    & \leqslant & \gamma \ 2^{\gamma j}  \ \left(\frac{\epsilon_2}{2^{s_2}}\right)^{p_+} |A_j|  \ , \label{newa}
\end{eqnarray}
provided
\[ \frac{e^{\tau_2} \ 2^{s_2 +1} }{\epsilon_2 } \leqslant  C_2 \ . \]
By applying the anisotropic embedding \eqref{PS}, for $\delta= p\frac{N+p_+}{N}$, one gets
\begin{eqnarray*}
 \left(\frac{\epsilon_2}{2^{s_2+j+2}}\right)^{p\frac{N+p_+}{N} }
|A_{j+1}| & \leqslant &  \iint_{Q_{j}} \left(\left( v-k_j\right)_-\xi\right)^{p_+} \ dz \ d\tau  \\
& \leqslant & \gamma \left( \sup_{ \frac{1}{2} \tau_2 \leqslant \tau \leqslant \tau_2  }\int_{\tilde{K}_{j}^{\frac{1}{C_2}}} \left(\left( v-k_j\right)_-\xi\right)^{p_+} \ dz
\right)^{\frac{p}{N}} \left( \sum_{i=1}^N \iint_{Q_{j}}  \left|((v-k_j)_-\xi)_{z_i}\right|^{p_i} \ dz \ d\tau  \right) \\
& \leqslant & \gamma \ 2^{\gamma j}  \ \left(\frac{\epsilon_2}{2^{s_2}}\right)^{p_+(\frac{p+N}{N})} \  \left(\frac{\epsilon_2}{2^{s_2}}\right)^{(p_+-2)\frac{p}{N}} |A_j|^{1+ \frac{p}{N}}
\end{eqnarray*}
and thereby, considering $Y_j= \frac{|A_j|}{|Q_j|}$,
\[ Y_{j+1} \leqslant \gamma \ 2^{\gamma j} \left(\frac{\epsilon_2}{2^{s_2}}\right)^{p_+-p} \ Y_j^{1+ \frac{p}{N}} \leqslant \gamma \ 2^{\gamma j} \ Y_j^{1+ \frac{p}{N}} \ . \]
We now choose $\nu = \gamma^{-\frac{N}{p}} \ 2^{-\gamma (\frac{N}{p})^2}$, depending only on the data, and conclude that
\[ v(z,\tau)\geqslant  \frac{\epsilon_2}{2^{s_2+1}} \ , \qquad (z,\tau)\in \tilde{K_2}^{\frac{1}{C_2}} \times \left(\frac{3}{4} \tau_2, \tau_2\right] \  \]
which implies that 
\begin{equation}\label{osc2alt}
u \leqslant \left(1- \frac{\epsilon_2}{2^{s_2+2}} e^{-\tau_2} \right) \omega \ , \qquad \mbox{in} \quad \tilde{K_{2r}}^{\frac{\omega}{2C_2}} \times (s+B_2 \left(\frac{\omega}{2}\right)^{1-\lambda_+} r^{p_+} ,  s+ \tilde{B_2}\left(\frac{\omega}{2}\right)^{1-\lambda_+} r^{p_+}]
\end{equation}
where
\[ B_2= \delta_2  e^{3\tau_{2}(p_+ - 2)/4} < \tilde{B_2}= \delta_2  e^{\tau_{2}(p_+ -2)} \ , \qquad C_2= \dfrac{2^{1+s_2}  \ e^{\tau_{2}}}{\epsilon_2} \ , \quad \tau_2= \left(\frac{2^{s_2}}{\epsilon_2}\right)^{p_+-2}\]
depend only on the data. 

\vspace{.2cm}

\noindent In what follows we combine the reduction of the oscillations of $u$ derived in the study of each one of the alternatives:
\begin{itemize}
    \item First alternative \\
    for $B_1= \delta_1  e^{3\tau_{1}(\lambda_+-1)/4} < \tilde{B_1}= \delta_1  e^{\tau_{1}(\lambda_+-1)}$, $C_1= \dfrac{2^{1+s_1}  \ e^{\tau_{1}}}{\epsilon_1}$, $\tau_1= \left(\frac{2^{s_1}}{\epsilon_1}\right)^{\lambda_+-1}$ and $\eta_1= 1- \frac{1}{2C_1}$

    \[ u \leq \eta_1 \ \omega , \quad \mbox{in} \quad K_{2r} \times (s+B_1 \left(\frac{\omega}{2}\right)^{1-\lambda_+} r^{p_+} ,  s+ \tilde{B_1}\left(\frac{\omega}{2}\right)^{1-\lambda_+} r^{p_+}]\]

    \item Second Alternative \\
    for $B_2= \delta_2  e^{3\tau_{2}(p_+ - 2)/4} < \tilde{B_2}= \delta_2  e^{\tau_{2}(p_+ -2)}$, $C_2= \dfrac{2^{1+s_2}  \ e^{\tau_{2}}}{\epsilon_2}$, $\tau_2= \left(\frac{2^{s_2}}{\epsilon_2}\right)^{p_+-2}$ and $\eta_2= 1- \frac{1}{2C_2}$
    
    \[ u \leq \eta_2 \  \omega , \quad \mbox{in} \quad K_{2r} \times (s+B_2 \left(\frac{\omega}{2}\right)^{1-\lambda_+} r^{p_+} ,  s+ \tilde{B_2}\left(\frac{\omega}{2}\right)^{1-\lambda_+} r^{p_+}]\]
    
\end{itemize}

\noindent We start by observing that one can take $\delta_2=\delta_1$, choosing $\delta_1$ smaller if needed, and $\epsilon_1=\epsilon_2$. 

\vspace{.2cm} 

\noindent Without loss of generality, assume that $p_+ <\lambda_+ +1$. Observe that, by taking $\epsilon_1$ even smaller and $s_2$ even bigger, if needed, 
\[ \frac{\tilde{B_1}}{B_2}  >1  \qquad \mbox{and} \qquad \frac{\tilde{B_2}}{\tilde{B_1}} >1 \ . \]
 With these choices and considering $ B=\max\{B_1, B_2\}$, we have that both alternatives are valid within the time interval 
\[ \left(s+B \left(\frac{\omega}{2}\right)^{1-\lambda_+} r^{p_+}, s+ \tilde{B_1}\left(\frac{\omega}{2}\right)^{1-\lambda_+} r^{p_+}\right] \]
and now we choose $s$ so that 
\[s + \tilde{B_1}\left(\frac{\omega}{2}\right)^{1-\lambda_+} r^{p_+}= A \left(\frac{\omega}{2}\right)^{1-\lambda_+} r^{p_+}, \qquad  A=\tilde{B_2} \ . \]

\noindent Therefore we obtain, for $\eta= \min\{ \eta_1, \eta_2\}$ and $Q= K_{2r} \times \left( (\tilde{B_2} - \tilde{B_1} + B) \left(\frac{\omega}{2}\right)^{1-\lambda_+} r^{p_+} ,  \tilde{B_2}\left(\frac{\omega}{2}\right)^{1-\lambda_+} r^{p_+}\right] $

\[ \essosc_Q  \ u  \leqslant \eta \ \omega \ .  \]

\vspace{.3cm}

\noindent As it is well known, results on the reduction of the oscillation will lead to the local H\"older continuity of $u$, once considered a proper sequence of nested and shrinking cylinders $Q_n$, all centered at $(x_o,t_o)$, and a decreasing sequence of positive numbers $\omega_n$ such that $ \displaystyle{\essosc_{Q_n}  \ u  \leqslant \ \omega_n}$, where  $Q_n$ tends to $(x_o,t_o)$ as $\omega_n$ goes to zero. 

\vspace{.1cm}

\noindent We intend to present a wider and more accurate picture of the local H\"older continuity for a more complete range of exponents in future work.

\section*{Acknowledgements} 

\noindent Simone Ciani is partially founded by GNAMPA (INdAM) and PNR italian fundings. Eurica Henriques was financed by Portuguese Funds through FCT - Funda\c c\~ao para a Ci\^encia e a Tecnologia - within the Projects UIDB/00013/2020 and UIDP/00013/2020 with the references DOI 10.54499/UIDB/00013/2020 (https://doi.org/ 10.54499/UIDB/00013/2020) and DOI 10.54499/UIDP/00013/2020 (https://doi.org/10.54499/ UIDP/00013/2020). Igor Skrypnik is partially supported by the grant "Mathematical modelling of complex systems and processes related to security’" of the National Academy of Sciences of Ukraine under the budget programme "Support for the development of priority areas of scientific research" for 2025- 2026, p/n 0125U000299. 

\vspace{.5cm}

\noindent{\bf Research Data Policy and Data Availability Statements.}  All data generated or analysed
during this study are included in this article.

\vspace{1cm}

CONTACT INFORMATION

\vspace{.5cm}

\noindent \textbf{Simone Ciani}\\
Department of Mathematics of the University of Bologna, Piazza Porta San Donato, 5, 40126 Bologna, Italy \\
 simone.ciani3@unibo.it

\vspace{.5cm}

\noindent \textbf{Eurica Henriques}\\
Centro de Matem\'atica, Universidade do Minho - Polo CMAT-UTAD \\
Departamento de Matem\'atica - Universidade de Tr\'as-os-Montes e Alto Douro, 5000-801 Vila Real, Portugal\\
eurica@utad.pt

\vspace{.5cm}

\noindent \textbf{Mariia Savchenko}\\
Institute of Applied Mathematics and Mechanics, National Academy of Sciences of Ukraine, Gen. Batiouk Str. 19, 84116 Sloviansk, Ukraine\\
shan\_maria@ukr.net

\vspace{.5cm}

\noindent \textbf{Igor I.~Skrypnik}\\Institute of Applied Mathematics and Mechanics,
National Academy of Sciences of Ukraine, \\ 
 Batiouk Str. 19, 84116 Sloviansk, Ukraine\\
Vasyl' Stus Donetsk National University, 600-richcha Str. 21, 21021 Vinnytsia, Ukraine\\
ihor.skrypnik@gmail.com

\end{document}